\documentclass[10pt,11pt]{amsart}%
\usepackage[dvips]{graphicx}
\usepackage{amsmath}
\usepackage{color}
\usepackage{amsfonts}
\usepackage{amssymb}%
\providecommand{\U}[1]{\protect\rule{.1in}{.1in}}
\providecommand{\U}[1]{\protect\rule{.1in}{.1in}}
\providecommand{\U}[1]{\protect\rule{.1in}{.1in}}
\providecommand{\U}[1]{\protect\rule{.1in}{.1in}}
\providecommand{\U}[1]{\protect\rule{.1in}{.1in}}
\providecommand{\U}[1]{\protect\rule{.1in}{.1in}}
\providecommand{\U}[1]{\protect\rule{.1in}{.1in}}
\newtheorem{theorem}{Theorem}[section]

\newtheorem{corollary}[theorem]{Corollary}

\newtheorem{definition}[theorem]{Definition}

\newtheorem{lemma}[theorem]{Lemma}

\newtheorem{proposition}[theorem]{Proposition}

\begin{document}
\title{Conditioned random walks from Kac-Moody root systems}
\date{December 20, 2013}
\author{C\'{e}dric Lecouvey, Emmanuel Lesigne and Marc Peign\'{e}}
\maketitle

\begin{abstract}
Random paths are time continuous interpolations of random walks. By using
Littelmann path model, we associate to each irreducible highest weight module
of a Kac Moody algebra ${\mathfrak{g}}$ a random path ${\mathcal{W}}.$ Under
suitable hypotheses, we make explicit the probability of the event $E$:
\textquotedblleft${\mathcal{W}}$ never exits the Weyl chamber of
${\mathfrak{g}}$\textquotedblright. We then give the law of the random walk
defined by ${\mathcal{W}}$ conditioned by the event $E$ and prove this law can
be recovered by applying to ${\mathcal{W}}$ a path transform of Pitman
type.\ This generalizes the main results of \cite{OC1} and \cite{LLP} to Kac
Moody root systems and arbitrary highest weight modules. Our approach here is
new and more algebraic that in \cite{OC1} and \cite{LLP}.\ We indeed fully
exploit the symmetry of our construction under the action of the Weyl group of
${\mathfrak{g}}$ which permits to avoid delicate generalizations of the
results of \cite{LLP} on renewal theory.

\end{abstract}

\section{Introduction}

The purpose of the paper is to study conditionings of random walks using
algebraic and combinatorial tools coming from representation theory of Lie
algebras and their infinite-dimensional generalizations (Kac-Moody algebras).
We extend in particular some results previously obtained in \cite{OC1},
\cite{OC2}, \cite{BBO}, \cite{LLP} and \cite{LLP2} to random paths in the
weight lattice of any Kac-Moody algebra ${\mathfrak{g}}$. To do this, we
consider a fixed ${\mathfrak{g}}$-module $V$ in the category ${\mathcal{O}%
}_{int}$ (a convenient generalization of the category of Lie algebras finite
dimensional representations). It decomposes as the direct sum of its weight
spaces, each such space being parametrized by a vector of the weight lattice
of ${\mathfrak{g}}$. The transitions of the random walk associated to $V$ are
then the weights of $V$.

The prototype of the results we obtain appears in the seminal paper \cite{OC1}
by O'Connell where it is shown that the law of the one-way simple random walk
$W$ in ${\mathbb{Z}}^{n}$ conditioned to stay in the cone ${\mathcal{C}%
}=\{(x_{1},\ldots,x_{n})\in{\mathbb{Z}}^{n}\mid x_{1}\geq\cdots\geq x_{n}%
\geq0\}$ and with drift in the interior $\mathring{{\mathcal{C}}}$ of
${\mathcal{C}}$, is the same as the law of a Markov chain $H$ obtained by
applying to $W$ a generalization of the Pitman transform.\ This transform is
defined via an insertion procedure on semistandard tableaux classically used
in representation theory of ${\mathfrak{sl}}_{n}({\mathbb{C}})$. The
transition matrix of $H$ can then be expressed in terms of the Weyl characters
(Schur functions) of the irreducible ${\mathfrak{sl}}_{n}({\mathbb{C}}%
)$-modules. Here the transitions of the random walk $W$ are the vectors of the
standard basis of ${\mathbb{Z}}^{n}$ which correspond to the weights of the
defining representation ${\mathbb{C}}^{n}$ of ${\mathfrak{sl}}_{n}%
({\mathbb{C)}}$. In addition to the insertion procedure on tableaux and some
classical facts about representation theory of ${\mathfrak{sl}}_{n}%
({\mathbb{C}})$, the main ingredients of O'Connell's result are a Theorem of
Doob on Martin boundary together with the asymptotic behavior of tensor
product multiplicities associated to the decompositions of $V^{\otimes\ell}$
in its irreducible components (which in this case are counted by standard skew tableaux).

We consider in \cite{LLP} more general random walks $W$ with transitions the
weights of a finite-dimensional irreducible ${\mathfrak{g}}$-module $V$ where
${\mathfrak{g}}$ is a Lie algebra.\ The law of $W$ is constructed so that the
probabilities of the paths only depend of their lengths and their ends.\ We
then show that the process $H$ obtained by applying to $W$ a generalization of
the Pitman transform introduced in \cite{BBO} is a Markov chain. When $V$ is a
minuscule representation (i.e. when the weights of $V$ belong to the same
orbit under the action of the Weyl group of ${\mathfrak{g}}$) and $W$ has
drift in the interior $\mathring{{\mathcal{C}}}$ of the cone ${\mathcal{C}}$
of dominant weights, we prove that $H$ has the same law as $W$ conditioned to
never exit ${\mathcal{C}}$. Similarly to the result of O'Connell, this common
law can be expressed in terms of the Weyl characters of the simple
${\mathfrak{g}}$-modules. Nevertheless the methods differ from \cite{OC1}
notably because there was no previously known asymptotic behavior for the
relevant tensor multiplicities in the more general cases we study. In fact, we
proceed by establishing a quotient renewal theorem for general random walks
conditioned to stay in a cone. When $W$ is not defined from a minuscule
representation, we also show that the law of $W$ conditioned to never exit
${\mathcal{C}}$ cannot coincide with that of $H$.

In \cite{LLP2}, we use the renewal theorem of \cite{LLP} and insertion
procedures on tableaux appearing in the representation theory of the Lie
superalgebras ${\mathfrak{gl}}(m,n)$ and ${\mathfrak{q}}(n)$ to extend the
results of \cite{OC1} to one way simple random walks conditioned to never exit
cones ${\mathcal{C}}^{\prime}$ for examples of cones ${\mathcal{C}}^{\prime}$
different from ${\mathcal{C}}$. \bigskip

In view of the results of \cite{LLP}, it is natural to ask whether the Markov
chain $H$ is related to a suitable conditioning of $W$ in the non minuscule
case. Also what can be said about the law of $W$ conditioned to never exit
${\mathcal{C}}$ ?\ In the sequel, we will answer both questions (partially for
the second) not only for random walks defined from representations of Lie
algebras but, more generally, for similar random walks with transitions the
weights of a highest weight module $V(\kappa)$ associated to a Kac-Moody
algebra ${\mathfrak{g}}$ of rank $n$.

\noindent By using Littelmann path model \cite{Lit2}, one can associate to
$V(\kappa)$ a countable set of piecewise continuous linear paths
$B(\pi_{\kappa})$ in the weight lattice $P\subset{\mathbb{R}}^{n}$ of
${\mathfrak{g}}$. These paths (called elementary in the sequel) are regarded
as functions $\pi:[0,1]\rightarrow{\mathbb{R}}^{n}$ such that $\pi(0)=0$ and
$\pi(1)\in P$. The weights of $V(\kappa)$ are then the elements $\pi(1),\pi\in
B(\pi_{\kappa}).$ The set $B(\pi_{\kappa})$ has the structure of a colored and
oriented graph isomorphic to the crystal graph of $V(\kappa)$ as defined by Kashiwara.

\noindent We use the crystal graph structure on $B(\pi_{\kappa})$ to endow it
as in \cite{LLP} with a probability density $p$. This yields a random variable
$X$ defined on $B(\pi_{\kappa})$ with probability distribution $p$.\ Let
$(X_{\ell})_{\ell\geq1}$ be a sequence of i.i.d. random variables with the
same law as $X$. We then define a continuous random path ${\mathcal{W}}$ such
that for any $t\geq0$, ${\mathcal{W}}(t)=X_{1}(1)+\cdots+X_{\ell-1}%
(1)+X_{\ell}(\ell-t)$ for any $t\in\lbrack\ell-1,\ell]$. The sequence
$W=(W_{\ell})_{\ell\geq0}$ defined by $W_{\ell}={\mathcal{W}}(\ell)$ is then a
random walk with transitions the weights of $V(\kappa)$ as considered in
\cite{LLP}.\ The main result of the paper is that, when $W$ has drift in
$\mathring{{\mathcal{C}}}$ (i.e. in the interior of the Weyl chamber of
${\mathfrak{g}}$), the law of its conditioning by the event $E=({\mathcal{W}%
}(t)\in{\mathcal{C}}$ for any $t\geq0)$ can be simply expressed in terms of
the Weyl-Kac characters. So the results of \cite{LLP} remain true for a
conditioning holding on the whole continuous trajectory (not only on its
discrete version at integer time). We also prove that the conditioned law so
obtained coincides with the law of the image of $W$ by the generalized Pitman
transform. When ${\mathfrak{g}}$ is finite-dimensional and $\kappa$ is
minuscule we recover in particular the main results of \cite{OC1} and
\cite{LLP}. On the representation theory side, our results also lead to
asymptotic behavior of tensor product multiplicities of Kac-Moody highest
weight modules.

\noindent Nevertheless our approach differ from that of \cite{LLP} since we do
not use any renewal theorem. Our strategy is more algebraic: we exploit the
symmetry of the representations with respect to the Weyl group ${\mathsf{W}}$
of ${\mathfrak{g}}$ and study simultaneously a family of random paths
${\mathcal{W}}^{w}$ indexed by the elements $w\in{\mathsf{W}}$. In particular
our proofs are independent of the results of \cite{OC1} and \cite{LLP}%
.\ \bigskip

The paper is organized as follows. In Section 2, we introduce the notions of
random walk and random path used in the paper. Section 3 recalls the necessary
background on Kac-Moody algebras and their representations and summarize some
important results on Littelmann's path model. The random path ${\mathcal{W}}$
and the random walk $W$ associated to $V(\kappa)$ are introduced in Section 4
together with the generalized Pitman transform and the Markov chain $H$. In
Section 5, we use a process of symmetrization to define the random paths
${\mathcal{W}}^{w},w\in{\mathsf{W}}$ from ${\mathcal{W=W}}^{1}$. This allows
us to give an explicit expression of the harmonic function $\mu\mapsto
{\mathbb{P}}_{\mu}({\mathcal{W}}(t){\mathcal{\in C}}$ for any $t\geq0)$ in
Section 6 and prove our main theorem. Its gives the probability that
${\mathcal{W}}$ starting at $\mu$ remains in ${\mathcal{C}}$. We also extend
it to the case of random walks defined from non irreducible representations of
simple Lie algebras.\ Finally Section 7 is devoted to additional results: we
give asymptotic behavior of tensor power multiplicities and also compare the
probabilities ${\mathbb{P}}_{\mu}({\mathcal{W}}(t){\mathcal{\in C}}$ for any
$t\geq0)$ and ${\mathbb{P}}_{\mu}(W_{\ell}\in{\mathcal{C}}$ for any $\ell
\geq0)$. \bigskip

\noindent\textbf{MSC classification:} 05E05, 05E10, 60G50, 60J10, 60J22.

\section{Random paths}

\subsection{Background on Markov chains}

\label{subsec-Markov} Consider a probability space $(\Omega,{\mathcal{F}%
},{\mathbb{P}})$ and a countable set $M$. A sequence $Y=(Y_{\ell})_{\ell\geq
0}$ of random variables defined on $\Omega$ with values in $M$ is a
\textit{Markov chain} when
\[
{\mathbb{P}}(Y_{\ell+1}=\mu_{\ell+1}\mid Y_{\ell}=\mu_{\ell},\ldots,Y_{0}%
=\mu_{0})={\mathbb{P}}(Y_{\ell+1}=\mu_{\ell+1}\mid Y_{\ell}=\mu_{\ell})
\]
for any any $\ell\geq0$ and any $\mu_{0},\ldots,\mu_{\ell},\mu_{\ell+1}\in M$.
The Markov chains considered in the sequel will also be assumed time
homogeneous, that is ${\mathbb{P}}(Y_{\ell+1}=\lambda\mid Y_{\ell}%
=\mu)={\mathbb{P}}(Y_{\ell}=\lambda\mid Y_{\ell-1}=\mu)$ for any $\ell\geq1$
and $\mu,\lambda\in M$.\ For all $\mu,\lambda$ in $M$, the transition
probability from $\mu$ to $\lambda$ is then defined by
\[
\Pi(\mu,\lambda)={\mathbb{P}}(Y_{\ell+1}=\lambda\mid Y_{\ell}=\mu)
\]
and we refer to $\Pi$ as the transition matrix of the Markov chain $Y$. The
distribution of $Y_{0}$ is called the initial distribution of the chain $Y$.

In the following, we will assume that $M$ is a subset of the euclidean space
${\mathbb{R}}^{n}$ for some $n\geq1$ and that the initial distribution of the
Markov chain $Y=(Y_{\ell})_{\ell\geq0}$ has full support, i.e. ${\mathbb{P}%
}(Y_{0}=\lambda)>0$ for any $\lambda\in M$. In \cite{LLP}, we have considered
a nonempty set ${\mathcal{C}}\subset M$ and an event $E\in{\mathcal{T}}$ such
that ${\mathbb{P}}(E\mid Y_{0}=\lambda)>0$ for all $\lambda\in{\mathcal{C}}$
and ${\mathbb{P}}(E\mid Y_{0}=\lambda)=0$ for all $\lambda\notin{\mathcal{C}}%
$; this implied that ${\mathbb{P}}(E)>0$, we could thus define the conditional
probability ${\mathbb{Q}}$ relative to this event: ${\mathbb{Q}}%
(\cdot):={\mathbb{P}}(\cdot|E)$. For example, we considered the event
$E:=(Y_{\ell}\in{\mathcal{C}}$ for any $\ell\geq0)$. In the present work we
will study more general situations, this involves to introduce some
generalities about continuous time Markov processes.

\textit{A continuous time Markov process} ${\mathcal{Y}}=({\mathcal{Y}%
}(t))_{t\geq0}$ on $(\Omega,{\mathcal{F}},{\mathbb{P}})$ with values in
${\mathbb{R}}^{n}$ is a family of random variables defined on $(\Omega
,{\mathcal{F}},{\mathbb{P}})$ such that, for any integer $k\geq1$, any $0\leq
t_{1}<\cdots<t_{k+1}$ and any Borel subsets $B_{1},\cdots,B_{k+1}$ of
${\mathbb{R}}^{n}$, one gets
\[
{\mathbb{P}}({\mathcal{Y}}(t_{k+1})\in B_{k+1}\mid Y(t_{1})\in B_{1}%
,Y(t_{2})\in B_{2},\cdots,Y(t_{k})\in B_{k})={\mathbb{P}}(Y(t_{k+1})\in
B_{k+1}\mid Y(t_{k})\in B_{k}).
\]
This is the Markov property, that we will use very often. In the following, we
shall need a more general version of this property which is a consequence of
the above. One can indeed show that for any $T\geq0$ and any Borel sets
$A\subset({\mathbb{R}}^{n})^{\otimes\lbrack0,T]},B\subset{\mathbb{R}}^{n}$ and
$C\subset({\mathbb{R}}^{n})^{\otimes\lbrack T,+\infty\lbrack}$, one gets
\[
{\mathbb{P}}(({\mathcal{Y}}(t))_{t\geq T}\in C\mid({\mathcal{Y}}(t))_{0\leq
t\leq T}\in A,{\mathcal{Y}}(T)\in B)={\mathbb{P}}(({\mathcal{Y}}(t))_{t\geq
T}\in C\mid{\mathcal{Y}}(T)\in B).
\]

In the sequel, we will assume the two following conditions.

\begin{enumerate}
\item For any integer $\ell\geq0$, one gets
\begin{equation}
Y_{\ell}:={\mathcal{Y}}(\ell)\in M\qquad{\mathbb{P}}{\mathrm{-almost\ surely}}
\label{discreteversion}%
\end{equation}
It readily follows that the sequence $Y=(Y_{\ell})_{\ell\geq0}$ is a
$M$-valued Markov chain.

\item For any $0\leq s\leq t$ and any Borel subsets $A,B\in{\mathbb{R}}^{n}$
\begin{equation}
{\mathbb{P}}({\mathcal{Y}}(t+1)\in B\mid{\mathcal{Y}}(s+1)\in A)={\mathbb{P}%
}({\mathcal{Y}}(t)\in B\mid{\mathcal{Y}}(s)\in A).
\label{invariancetranslation}%
\end{equation}
Combining this condition with the Markov property, one checks that for any
$T\geq1$ and $x\in{\mathbb{R}}^{n}$, the conditional distribution of the
process $({\mathcal{Y}}(t+1))_{t\geq T}$ with respect to the event
$({\mathcal{Y}}(T+1)=x)$ is equal to the one of $({\mathcal{Y}}(t))_{t\geq T}$
with respect to $({\mathcal{Y}}(T)=x)$.
\end{enumerate}

In the following, we will assume that the initial distribution of the Markov
process $({\mathcal{Y}}(t))_{t\geq0}$ has full support, i.e. ${\mathbb{P}%
}({\mathcal{Y}}(0)=\lambda)>0$ for any $\lambda\in M$. We will also consider a
nonempty set ${\mathcal{C}}\subset{\mathbb{R}}^{n}$ and will assume that the
probability of the event $E:=({\mathcal{Y}}(t)\in{\mathcal{C}}%
\ {\mathrm{for\ any\ }}t\geq0)$ is positive; the conditional probability
${\mathbb{Q}}$ relative to $E$ is thus well defined. The following proposition
can be deduced from our hypotheses and the Markov property of $Y$. We postpone
its proof to the appendix.

\begin{proposition}
\label{Prop_Q}Let $({\mathcal{Y}}(t))_{t\geq0}$ be a continuous time Markov
process with values in ${\mathbb{R}}^{n}$ satisfying conditions
(\ref{discreteversion}) and (\ref{invariancetranslation}) and ${\mathcal{C}%
}\subset{\mathbb{R}}^{n}$ such that the event $E:=({\mathcal{Y}}%
(t)\in{\mathcal{C}}\ {\mathrm{for\ any\ }}t\geq0)$ has positive probability
measure. Then, under the probability ${\mathbb{Q}}(\cdot)={\mathbb{P}}%
(\cdot|E)$, the sequence $(Y_{\ell})_{\ell\geq0}$ is still a Markov chain with
values in ${\mathcal{C}}\cap M$ and transition probabilities given by
\begin{equation}
\forall\mu,\lambda\in{\mathcal{C}}\cap M\quad{\mathbb{Q}}(Y_{\ell+1}%
=\lambda\mid Y_{\ell}=\mu)=\Pi^{E}(\mu,\lambda)\frac{{\mathbb{P}}(E\mid
Y_{0}=\lambda)}{{\mathbb{P}}(E\mid Y_{0}=\mu)} \label{reco}%
\end{equation}
where $\Pi^{E}(\mu,\lambda)={\mathbb{P}}(Y_{\ell+1}=\lambda,{\mathcal{Y}%
}(t)\in{\mathcal{C}}$ for $t\in\lbrack\ell,\ell+1]\mid Y_{\ell}=\mu)$. We will
denote by $Y^{E}$ this Markov chain
\end{proposition}

To simplify the notations we will denote by ${\mathcal{C}}$ the set
${\mathcal{C}}\cap M$ as soon as we will consider the Markov chain $(Y_{\ell
})_{\ell\geq0}$ and $\Pi^{E}=(\Pi(\mu,\lambda))_{\mu,\lambda\in{\mathcal{C}}}$
the \textquotedblleft restriction\textquotedblright\ of the transition matrix
$\Pi$ to the event $E$ where
\[
\Pi^{E}(\mu,\lambda)={\mathbb{P}}(Y_{\ell+1}=\lambda,{\mathcal{Y}}%
(t)\in{\mathcal{C}}{\text{ for }}t\in\lbrack\ell,\ell+1]\mid Y_{\ell}=\mu).
\]
So $\Pi^{E}(\mu,\lambda)$ gives the probability of the transition from $\mu$
to $\lambda$ when ${\mathcal{Y}}(t)$ remains in ${\mathcal{C}}$ for
$t\in\lbrack\ell,\ell+1]$. \bigskip

A \emph{substochastic matrix} on the countable set $M$ is a map $\Pi:M\times
M\rightarrow\lbrack0,1]$ such that $\sum_{y\in M}\Pi(x,y)\leq1$ for any $x\in
M.\;$If $\Pi,\Pi^{\prime}$ are substochastic matrices on $M$, we define their
product $\Pi\times\Pi^{\prime}$ as the substochastic matrix given by the
ordinary product of matrices:
\[
\Pi\times\Pi^{\prime}(x,y)=\sum_{z\in M}\Pi(x,z)\Pi^{\prime}(z,y).
\]

A function $h:M\rightarrow{\mathbb{R}}$ is \emph{harmonic} for the
substochastic transition matrix $\Pi$ when we have $\sum_{y\in M}%
\Pi(x,y)h(y)=h(x)$ for any $x\in M$. Consider a (strictly) positive harmonic
function $h$. We can then define the Doob transform of $\Pi$ by $h$ (also
called the $h$-transform of $\Pi$) setting
\[
\Pi_{h}(x,y)=\frac{h(y)}{h(x)}\Pi(x,y).
\]
We then have $\sum_{y\in M}\Pi_{h}(x,y)=1$ for any $x\in M.\;$Thus $\Pi_{h}$
is stochastic and can be interpreted as the transition matrix for a certain
Markov chain.

An example is given in formula (\ref{reco}): the state space is now
${\mathcal{C}}$, the substochastic matrix is $\Pi^{E}$ and the harmonic
function is $h_{E}(\mu):={\mathbb{P}}(E\mid Y_{0}=\mu)$; the transition matrix
$\Pi_{h_{E}}^{E}$ is the transition matrix of the Markov chain $Y^{E}$.

\subsection{Elementary random paths}

\label{subsec-ele_path} Consider a ${\mathbb{Z}}$-lattice $P$ with finite rank
$d$. Set $P_{{\mathbb{R}}}=P\otimes_{{\mathbb{Z}}}{\mathbb{R}}$ so that $P$
can be regarded as a ${\mathbb{Z}}$-lattice of rank $d$ in ${\mathbb{R}}^{d}%
$.\ An \emph{elementary path} is a piecewise continuous linear map
$\pi:[0,1]\rightarrow P_{{\mathbb{R}}}$ such that $\pi(0)=0$ and $\pi(1)\in
P$.\ Two paths $\pi_{1}$ and $\pi_{2}$ are considered as identical if there
exists a piecewise, surjective continuous and nondecreasing map
$u:[0,1]\rightarrow\lbrack0,1]$ such that $\pi_{2}=\pi_{1}\circ u$.

The set ${\mathcal{F}}$ of continuous functions from $[0,1]$ to
$P_{{\mathbb{R}}}$ is equipped with the norm $\left\Vert \cdot{}\right\Vert
_{\infty}$ of uniform convergence : for any $\pi\in{\mathcal{F}}$, on has
$\left\Vert \pi\right\Vert :=\sup_{t\in\lbrack0,1]}\left\Vert \pi
(t)\right\Vert _{2}$ where $\left\Vert \cdot{}\right\Vert _{2}$ denotes the
euclidean norm on ${\mathbb{R}}^{d}$. Let $B$ be a \emph{countable set of
paths} and fix a probability distribution $p=(p_{\pi})_{\pi\in B}$ on $B$ such
that $p_{\pi}>0$ for any $\pi\in B$. Let $X$ be a random variable defined on a
probability space $(\Omega,{\mathcal{F}},{\mathbb{P}})$ and with distribution
$p$ (in other words ${\mathbb{P}}(X=\pi)=p_{\pi}\ {\text{for any }}\pi\in B).$
{The variable $X$ admits a moment of order $1$ (namely ${\mathbb{E}}(\Vert
X\Vert)<+\infty$) when the series of functions $\sum_{\pi}p_{\pi}\left\Vert
\pi\right\Vert $ converges on $[0,1]$. We then set%
\[
m:={\mathbb{E}}(X)=\sum_{\pi\in B}p_{\pi}\pi.
\]
}

The concatenation $\pi_{1}\ast\pi_{2}$ of two elementary paths $\pi_{1}$ and
$\pi_{2}$ is defined by
\[
\pi_{1}\ast\pi_{2}(t)=\left\{
\begin{array}
[c]{lll}%
\pi_{1}(2t) & {\text{ for }} & t\in\lbrack0,\frac{1}{2}],\\
\pi_{1}(1)+\pi_{2}(2t-1) & {\text{ for }} & t\in\lbrack\frac{1}{2},1].
\end{array}
\right.
\]

In the sequel, ${\mathcal{C}}$ is a closed convex cone in $P_{{\mathbb{R}}}$
with interior $\mathring{{\mathcal{C}}}$ and we set $P_{+}={\mathcal{C}}\cap
P$.

\subsection{Random paths}

\label{subsec_RandPath}Let $B$ be a set of elementary paths and $(X_{\ell
})_{\ell\geq1}$ a sequence of i.i.d. random variables with law $X$ where $X$
is the random variable with values in $B$ introduced in \ref{subsec-ele_path}.
We define the random process ${\mathcal{W}}$ as follows: for any $\ell
\in{\mathbb{Z}}_{>0}$ and $t\in\lbrack\ell,\ell+1]$
\[
{\mathcal{W}}(t):=X_{1}(1)+X_{2}(1)+\cdots+X_{\ell-1}(1)+X_{\ell}(t-\ell).
\]
The sequence of random variables $W=(W_{\ell})_{\ell\geq0}:=({\mathcal{W}%
}(\ell))_{\ell\geq0}$ is a random walk with set of increments $I:=\{\pi
(1)\mid\pi\in B\}$.

For any $\ell\geq1$, let $\psi_{\ell}$ be the map defined by
\[
\forall\mu\in{\mathcal{C}}\quad\psi_{\ell}(\mu)={\mathbb{P}}_{\mu
}({\mathcal{W}}(t)\in{\mathcal{C}}\ {\text{ for any }}t\in\lbrack0,\ell])
\]
so that $\psi_{\ell}(\mu)$ is the probability that ${\mathcal{W}}$ starting at
$\mu$ remains in ${\mathcal{C}}$ for any $t\in\lbrack0,\ell]$.\ As
$\ell\rightarrow+\infty$, the sequence of functions $(\psi_{\ell})_{\ell\geq
0}$ converges to the function $\psi$ defined by
\[
\forall\mu\in{\mathcal{C}}\quad\psi_{{}}(\mu)={\mathbb{P}}_{\mu}({\mathcal{W}%
}(t)\in{\mathcal{C}}\ {\text{ for any }}t\geq0).
\]

\begin{proposition}
{\label{Prop_lim}Assume ${\mathbb{E}}(\Vert X\Vert)<+\infty$ and
$m(1)\notin\mathring{\mathcal{C}}$. Then for any $\mu\in{\mathcal{C}}$, we
have $\psi(\mu)=0$.}
\end{proposition}

\begin{proof}
Observe that $\psi(\mu)={\mathbb{P}}_{\mu}({\mathcal{W}}(t)\in{\mathcal{C}}$
for any $t\geq0)\leq{\mathbb{P}}_{\mu}(W_{\ell}\in{\mathcal{C}}$ for any
$\ell\geq0)$.\ By a straightforward application of the strong law of large
numbers for the random walk $W$ (see \cite{LLP} for more details), we have
${\mathbb{P}}_{\mu}(W_{\ell}\in{\mathcal{C}}$ for any $\ell\geq0)=0$ when
$m(1)\notin\mathring{\mathcal{C}}$.\ Thus $\psi(\mu)=0$ when $m(1)\notin
\mathring{\mathcal{C}}$.
\end{proof}

\bigskip

\noindent\textbf{Remark:} The hypothesis $\mathbb{E}(\left\Vert X\right\Vert
)<+\infty$ suffices in fact to prove also that {$\psi(\mu)>0$ when
$m(1)\in\mathring{\mathcal{C}}$ }and there exists at least $\pi\in B$ such
that $\operatorname{Im}\pi\subset{\mathcal{C}}$. In the context of the paper,
this will readily follows from Theorem \ref{Th_PSi} so we do not pursue in
this direction.

\section{Representations of symmetrizable Kac-Moody algebras}

\subsection{Symmetrizable Kac-Moody algebras}

Let $A=(a_{i,j})$ be a $n\times n$ generalized Cartan matrix of rank
$r$.\ This means that the entries $a_{i,j}\in{\mathbb{Z}}$ satisfy the
following conditions

\begin{enumerate}
\item $a_{i,j}\in{\mathbb{Z}}$ for $i,j\in\{1,\ldots,n\},$

\item $a_{i,i}=2$ for $i\in\{1,\ldots,n\},$

\item $a_{i,j}=0$ if and only if $a_{j,i}=0$ for $i,j\in\{1,\ldots,n\}$.
\end{enumerate}

We will also assume that $A$ is \emph{indecomposable}: given subsets $I$ and
$J$ of $\{1,\ldots,n\}$, there exists $(i,j)\in I\times J$ such that
$a_{i,j}\not =0$.\ We refer to \cite{KacB} for the classification of
indecomposable generalized Cartan matrices.\ Recall there exist only three
kinds of such matrices: when all the principal minors of $A$ are positive, $A$
is of \emph{finite type} and corresponds to the Cartan matrix of a simple Lie
algebra over ${\mathbb{C}}$; when all the proper principal minors of $A$ are
positive and $\det(A)=0$ the matrix $A$ is said of\emph{ affine type};
otherwise $A$ is of \emph{indefinite type}. For technical reasons, from now
on, we will restrict ourselves to \emph{symmetrizable} generalized Cartan
matrices i.e. we will assume there exists a diagonal matrix $D$ with entries
in ${\mathbb{Z}}_{>0}$ such that $DA$ is symmetric.

The root and weight lattices associated to a generalized symmetrizable Cartan
matrix are defined by mimic the construction for the Lie algebras.\ Let
$P^{\vee}$ be a free abelian group of rank $2n-r$ with ${\mathbb{Z}}$-basis
$\{h_{1},\ldots,h_{n}\}\cup\{d_{1},\ldots,d_{n-r}\}$.\ Set ${\mathfrak{h}%
}:=P^{\vee}\otimes_{{\mathbb{Z}}}{\mathbb{C}}$ and ${\mathfrak{h}%
}_{{\mathbb{R}}}:=P^{\vee}\otimes_{{\mathbb{Z}}}{\mathbb{R}}$. The weight
lattice $P$ is then defined by
\[
P:=\{\gamma\in{\mathfrak{h}}^{\ast}\mid\gamma(P^{\vee})\subset{\mathbb{Z\}}%
}{\text{.}}
\]
Set $\Pi^{\vee}:=\{h_{1},\ldots,h_{n}\}.\;$One can then choose a set
$\Pi:=\{\alpha_{1},\ldots,\alpha_{n}\}$ of linearly independent vectors in
$P\subset{\mathfrak{h}}^{\ast}$ such that $\alpha_{i}(h_{j})=a_{i,j}$ for
$i,j\in\{1,\ldots,n\}$ and $\alpha_{i}(d_{j})\in\{0,1\}$ for $i\in
\{1,\ldots,n-r\}$.\ The elements of $\Pi$ are the \emph{simple roots}.\ The
free abelian group $Q:=\bigoplus_{i=1}^{n}{\mathbb{Z}}\alpha_{i}$ is the
\emph{root lattice}.\ The quintuple $(A,\Pi,\Pi^{\vee},P,P^{\vee})$ is called
a \emph{generalized Cartan datum} associated to the matrix $A$.\ For any
$i=1,\ldots,n$, we also define the\emph{ fundamental weight} $\omega_{i}\in P$
by $\omega_{i}(h_{j})=\delta_{i,j}$ for $j\in\{1,\ldots,n\}$ and $\omega
_{i}(d_{j})=0$ for $j\in\{1,\ldots,n-r\}$.\ 

For any $i=1,\ldots,n$, we define the simple reflection $s_{i}$ on
${\mathfrak{h}}^{\ast}$ by
\begin{equation}
s_{i}(\gamma)=\gamma-h_{i}(\gamma)\alpha_{i}\ {\text{ for any }}\gamma\in
P{\text{.}} \label{defSi}%
\end{equation}
The \emph{Weyl group} ${\mathsf{W}}$ is the subgroup of $GL({\mathfrak{h}%
}^{\ast})$ generated by the reflections $s_{i}$.\ Each element $w\in
{\mathsf{W}}$ admits a reduced expression $w=s_{i_{1}}\cdots s_{i_{r}}$.\ One
can prove that $r$ is independent of the reduced expression considered so the
signature $\varepsilon(w)=(-1)^{r}$ is well-defined.

\begin{definition}
The \emph{Kac-Moody algebra} ${\mathfrak{g}}$ associated to the quintuple
$(A,\Pi,\Pi^{\vee},P,P^{\vee})$ is the ${\mathbb{C}}$-algebra generated by the
elements $e_{i},f_{i},$ $i=1,\ldots,n$ and $h\in P$ together with the relations

\begin{enumerate}
\item $[h,h^{\prime}]=0$ for any $h,h^{\prime}\in P$,

\item $[h,e_{i}]=\alpha_{i}(h)e_{i}$ for any $i=1,\ldots,n$ and $h\in P$,

\item $[h,f_{i}]=-\alpha_{i}(h)f_{i}$ for any $i=1,\ldots,n$ and $h\in P$,

\item $[e_{i},f_{j}]=\delta_{i,j}h_{i}$ for any $i,i=1,\ldots,n$,

\item $ad(e_{i})^{1-a_{i,j}}(e_{j})=0$ for any $i,j=1,\ldots,n$ such that
$i\neq j$,

\item $ad(f_{i})^{1-a_{i,j}}(f_{j})=0$ for any $i,j=1,\ldots,n$ such that
$i\neq j$,

where $ad(a)\in End{\mathfrak{g}}$ is defined by $ad(a)(b)=[a,b]:=ab-ba$ for
any $a,b\in{\mathfrak{g}}$.
\end{enumerate}
\end{definition}

Denote by ${\mathfrak{g}}_{+}$ and ${\mathfrak{g}}_{-}$ the subalgebras of
${\mathfrak{g}}$ generated by the $e_{i}$'s and the $f_{i}$'s,
respectively.\ We have the triangular decomposition ${\mathfrak{g=g}}%
_{+}\oplus{\mathfrak{h}}\oplus{\mathfrak{g}}_{-}$ and ${\mathfrak{h}}$ is
called the Cartan subalgebra of ${\mathfrak{g}}$.\ For any $\alpha\in Q$, set
\[
{\mathfrak{g}}_{\alpha}:=\{x\in{\mathfrak{g}}\mid\lbrack h,x]=\alpha
(h)x{\text{ for any }}h\in{\mathfrak{h}}\}.
\]
The algebra ${\mathfrak{g}}$ then decomposes on the form
\[
{\mathfrak{g=}}\bigoplus_{\alpha\in Q}{\mathfrak{g}}_{\alpha}
\]
where $\operatorname{dim}{\mathfrak{g}}_{\alpha}$ is finite for any $\alpha\in
Q$.\ The roots of ${\mathfrak{g}}$ are the nonzero elements $\alpha\in Q$ such
that ${\mathfrak{g}}_{\alpha}\neq\{0\}$.\ We denote by $R$ the set of roots of
${\mathfrak{g}}$. Set $Q_{+}:=\bigoplus_{i=1}^{n}{\mathbb{Z}}_{\geq0}%
\alpha_{i},$ $R_{+}:=R\cap Q_{+}$ and $R_{-}=R\cap(-Q_{+})$. Then one can
prove that $R=R_{+}\cup R_{-}$ and $R_{-}=-R_{+}$ as for the finite
dimensional Lie algebras.\ For any $\gamma=\sum_{i=1}^{n}a_{i}\alpha_{i}\in
Q_{+}$, we set
\[
ht(\gamma):=\sum_{i=1}^{n}a_{i}.
\]
We have the decomposition
\[
{\mathfrak{g=}}\bigoplus_{\alpha\in R_{+}}{\mathfrak{g}}_{\alpha}%
\oplus{\mathfrak{h}}\oplus\bigoplus_{\alpha\in R_{-}}{\mathfrak{g}}_{\alpha}.
\]
For any $\alpha\in R_{+}$, we set $\operatorname{dim}{\mathfrak{g}}_{\alpha
}=m_{\alpha}$ the multiplicity of the root $\alpha$ in ${\mathfrak{g}}$.\ The
set $R_{+}$ is infinite as soon as $A$ is not of finite type; the multiplicity
$m_{\alpha}$ may be greater than $1$ but is always bounded as follows (see
\cite{KacB} \S \ 1.3):
\begin{equation}
m_{\alpha}\leq n^{ht(\alpha)}\ {\text{ for any }}\alpha\in R_{+}.
\label{mag_mult}%
\end{equation}
When $A$ is not of finite type, the Weyl group ${\mathsf{W}}$ is also infinite
and there exist roots $\alpha\in R$ which do not belong to any orbit
${\mathsf{W}}\alpha_{i},i=1,\ldots,n$ of a simple root; these roots are called
\emph{imaginary roots} in contrast to \textit{real roots} which belong to the
orbit of a simple root $\alpha_{i}$.

The root system associated to a matrix $A$ of finite type is well known (see
for instance \cite{Bour}) and are classified in four infinite series
($A_{n},B_{n},C_{n}$ and $D_{n}$) and five exceptional systems ($E_{6}%
,E_{7},E_{8},F_{4},G_{2}$). In contrast, few is known on the root system
associated to a matrix of indefinite type. In the intermediate case of the
affine matrices, there also exists a finite classification which makes appear
seven infinite series and seven exceptional systems. The root system can be
described as follows. First, the rows and columns of $A$ can be ordered such
that the submatrix $A^{\circ}$ of size $(n-1)\times(n-1)$ obtained by deleting
the row and column indexed by $n$ in $A$ is the Cartan matrix of a finite root
system $R^{\circ}$. The kernel of $A$ has dimension 1; more precisely, there
exists a unique $n$-tuple $(a_{1},\ldots,a_{n})$ of positive relatively prime
integers such that $A^{t}(a_{1},\ldots,a_{n})=0$ and the vector $\delta
=\sum_{i=1}^{n}a_{i}\alpha_{i}$ then belongs to $R$. The sets of real roots,
of imaginary roots, of positive real roots and positive imaginary roots can be
completely described in terms of roots in $R^{\circ}$ and $\delta$. We refer
to \cite{KacB} p. 83 for a complete exposition and only recall the following
facts we need in the sequel.\ In particular, we do not need the complete
description of the sets $R_{+}^{re}$ which strongly depends on the affine root
system considered. We have
\[
R_{+}^{re}\subset\{\alpha+k\delta\mid\alpha\in R^{\circ},k\in{\mathbb{Z}}%
_{>0}{\mathbb{\}\cup}}R_{+}^{\circ}
\]
except for the affine root system $A_{2n}^{(2)}$ in which case
\[
R_{+}^{re}\subset\{\alpha+k\delta\mid\alpha\in R^{\circ},k\in{\mathbb{Z}}%
_{>0}{\mathbb{\}\cup\{}}\frac{1}{2}(\alpha+(2k-1)\delta\mid\alpha\in R^{\circ
},k\in Z_{>0}{\mathbb{{\mathbb{\}}}\cup}}R_{+}^{\circ}.
\]
We also have in all affine cases
\begin{equation}
R_{+}^{im}=\{k\delta\mid k\in{\mathbb{Z}}_{>0}{\mathbb{\}}}\quad{\text{and}%
}\quad R_{+}=R_{+}^{re}\cup R_{+}^{im}{\text{.}} \label{PosAffRoots}%
\end{equation}
The multiplicities of the positive roots verify (see \cite{KacB} Corollary
8.3).
\begin{equation}
m_{\alpha}=1{\text{ for }}\alpha\in R_{+}^{re}{\text{ and }}m_{\alpha}\leq
n{\text{ for }}\alpha\in R_{+}^{im}. \label{multiaffine}%
\end{equation}

\subsection{The category ${\mathcal{O}}_{int}$ of ${\mathfrak{g}}$-modules}

Let ${\mathfrak{g}}$ be a symmetrizable Kac-Moody algebra.\ We now introduce a
category of ${\mathfrak{g}}$-modules whose properties naturally extend those
of the finite-dimensional representations of simple Lie algebras.

\begin{definition}
The category ${\mathcal{O}}_{int}$ is the category of ${\mathfrak{g}}$-modules
$M$ satisfying the following properties:

\begin{enumerate}
\item The module $M$ decomposes in weight subspaces on the form
\[
M=\bigoplus_{\gamma\in P}M_{\gamma}{\text{ where }}M_{\gamma}:=\{v\in M\mid
h(v)=\gamma(h)v\ {\text{ for any }}h\in{\mathfrak{h}}\}.
\]

\item For any $i=1,\ldots,n$, the actions of $e_{i}$ and $f_{i}$ are locally
nilpotent i.e. for any $v\in M$, there exists integers $p$ and $q$ such that
$e_{i}^{p}\cdot v=f_{i}^{q}\cdot v=0$.
\end{enumerate}
\end{definition}

For any $\gamma\in P$, let $e^{\gamma}$ be the generator of the group algebra
${\mathbb{C}}[P]$ associated to $\gamma$. By definition, we have $e^{\gamma
}e^{\gamma^{\prime}}=e^{\gamma+\gamma^{\prime}}$ for any $\gamma
,\gamma^{\prime}\in P$ and the group ${\mathsf{W}}$ acts on ${\mathbb{C}}[P]$
as follows: $w(e^{\gamma})=e^{w(\gamma)}$ for any $w\in{\mathsf{W}}$ and any
$\gamma\in P$.

The irreducible modules in the category ${\mathcal{O}}_{int}$ are the
irreducible highest weight modules, they are parametrized by the
\emph{integral cone of dominant weights } $P_{+}$ of ${\mathfrak{g}}$ defined
by
\[
P_{+}:=\{\lambda\in P\mid\lambda(h_{i})\geq0\ {\text{ for any }}%
i=1,\ldots,n\}.
\]
The irreducible highest weight module $V(\lambda)$ of weight $\lambda\in
P_{+}$ decomposes as $V(\lambda)=\bigoplus_{\gamma\in P}V(\lambda)_{\gamma};$
observe that $\operatorname{dim}V(\lambda)$ is infinite when ${\mathfrak{g}}$
is not of finite type, nevertheless the weight space $V(\lambda)_{\gamma}$ is
always finite-dimensional and we set $K_{\lambda,\gamma}:= \operatorname{dim}%
(V(\lambda)_{\gamma})$. Furthermore, we have $\operatorname{dim}%
V(\lambda)_{\lambda}=1$ and $e_{i}(v)=0$ for any $i=1,\ldots,n$ and $v\in
V(\lambda)_{\lambda}$; the elements of $V(\lambda)_{\lambda}$ thus coincide up
to a multiplication by a scalar and are called the \textit{highest weight
vectors}.

The character $s_{\lambda}$ of $V(\lambda)$ is defined by $s_{\lambda}%
:=\sum_{\gamma\in P}K_{\lambda,\gamma}e^{\gamma};$ it is invariant under the
action of the Weyl group ${\mathsf{W}}$ since $K_{\lambda,\gamma}%
=K_{\lambda,w(\gamma)}$ for any $w\in{\mathsf{W}}$. Observe that the orbit
${\mathsf{W}}\cdot\gamma$ intersects $P_{+}$ exactly once when $K_{\lambda
,\gamma}>0$.

From now on, we fix a weight $\rho\in P$ such that $\rho(h_{i})=1$ for any
$i=1,\ldots,n$; we have the Kac-Weyl character formula :

\begin{theorem}
For any $\lambda\in P_{+}$, we have $\displaystyle s_{\lambda}=\frac
{\sum_{w\in{\mathsf{W}}}\varepsilon(w)e^{w(\lambda+\rho)-\rho}}{\prod
_{\alpha\in R_{+}}(1-e^{-\alpha})^{m_{\alpha}}}\label{WKCF}.$
\end{theorem}

The category ${\mathcal{O}}_{int}$ is stable under the tensor product of
${\mathfrak{g}}$-modules.\ Moreover, every module $M\in{\mathcal{O}}_{int}$
decomposes has a direct sum of irreducible modules.\ Given $\lambda
^{(1)},\ldots,\lambda^{(k)}$ a sequence of dominant weights, consider the
module $M:=V(\lambda^{(1)})\otimes\cdots\otimes V(\lambda^{(r)})$.\ Then
$\operatorname{dim}M_{\gamma}$ is finite for any $\gamma\in P$, the character
of $M$ can be defined by ${\mathrm{char}}(M):=\sum_{\gamma\in P}%
\operatorname{dim}M_{\gamma}e^{\gamma}$ and we have
\[
{\mathrm{char}}(M)=s_{\lambda^{(1)}}\cdots s_{\lambda^{(r)}.}%
\]
Each irreducible component of $M$ appears finitely many times in this
decomposition, in other words there exist nonnegative integers $m_{M,\lambda}$
such that
\[
M\simeq\bigoplus_{\lambda\in P_{+}}V(\lambda)^{\oplus m_{M,\lambda}}{\text{ or
equivalently }}{\mathrm{char}}(M):=\sum_{\lambda\in P_{+}}m_{M,\lambda
}s_{\lambda}.
\]
Consider $\kappa,\mu\in P_{+}$ and $\ell\in{\mathbb{Z}}_{\geq0}$.\ We set
\begin{equation}
V(\mu)\otimes V(\kappa)^{\otimes\ell}=\sum_{\lambda\in P_{+}}V(\lambda
)^{\oplus f_{\lambda/\mu}^{\kappa,\ell}}{\text{ and }}m_{\mu,\kappa}^{\lambda
}=f_{\lambda/\mu}^{\kappa,1}. \label{def_multi}%
\end{equation}
In the sequel, we will fix $\kappa\in P_{+}$ and write $f_{\lambda/\mu
}^{\kappa,\ell}=f_{\lambda/\mu}^{\ell}$ for short.

\subsection{Littelmann path model}

The aim of this paragraph is to give a brief overview of the path model
developed by Littelmann and its connections with Kashiwara crystal basis
theory.\ We refer to \cite{Lit1}, \cite{Lit2}, \cite{Lit3} and \cite{Kashi}
for examples and a detailed exposition. Let ${\mathfrak{g}}$ be a
symmetrizable Kac-Moody algebra associated to the quintuple $(A,\Pi,\Pi^{\vee
},P,P^{\vee})$ where $A$ is a $n\times n$ symmetrizable generalized Cartan
matrix with rank $r$. In the following, it will be convenient to fix a
nondegenerate symmetric bilinear form $\langle\cdot,\cdot\rangle$ on
${\mathfrak{h}}_{{\mathbb{R}}}^{\ast}$ invariant under ${\mathsf{W}}$.\ For
any root $\alpha$, we set $\alpha^{\vee}=\frac{\alpha}{\langle\alpha
,\alpha\rangle}$. We have seen that $P$ is a ${\mathbb{Z}}$-lattice with rank
$d=2n-r$.\ We define the notion of elementary piecewise linear paths in
$P_{{\mathbb{R}}}:=P\otimes_{{\mathbb{Z}}}{\mathbb{R}}$ as we did in
\S \ \ref{subsec-ele_path}.\ Let ${\mathcal{P}}$ be the set of such elementary
paths having \emph{only rational turning points} (i.e. whose inflexion points
have rational coordinates) and ending in $P$ i.e. such that $\pi(1)\in P$. The
Weyl group ${\mathsf{W}}$ acts on ${\mathcal{P}}$ as follows: for any
$w\in{\mathsf{W}}$ and $\eta\in{\mathcal{P}}$, the path $w[\eta]$ is defined
by
\begin{equation}
\forall t\in\lbrack0,1]\qquad w[\eta](t)=w(\eta(t)) \label{actW}%
\end{equation}
and the weight ${\mathrm{wt}}(\eta)$ of $\eta$ is defined by ${\mathrm{wt}%
}(\eta)=\eta(1)$.\ 

We now define operators $\tilde{e}_{i}$ and $\tilde{f}_{i},$ $i=1,\ldots,n,$
acting on ${\mathcal{P}}\cup\{{\mathbf{0}}\}$.\ If $\eta={\mathbf{0}}$, we set
$\tilde{e}_{i}(\eta)=\tilde{f}_{i}(\eta)={\mathbf{0}}$; when $\eta
\in{\mathcal{P}}$, we need to decompose $\eta$ into a union of finitely many
subpaths and reflect some of these subpaths by $s_{\alpha_{i}}$ according to
the behavior of the map
\[
h_{\eta}:\left\{
\begin{array}
[c]{ccl}%
\lbrack0,1] & \rightarrow & {\mathbb{R}}\\
t & \mapsto & \langle\eta(t),\alpha_{i}^{\vee}\rangle.
\end{array}
\right.
\]
Let $m_{\eta}$ for the minimum of the function $h_{\eta}$.\ Since $h_{\eta
}(0)=0$, we have $m_{\eta}\leq0$. \bigskip

If $m_{\eta}>-1$, then $\tilde{e}_{i}(\eta)={\mathbf{0}}$. If $m_{\eta}\leq
-1$, set $t_{1}:= \operatorname{inf}\{t\in\lbrack0,1]\mid h_{\eta}(t)=m_{\eta
}\}$ and let $t_{0}\in\lbrack0,t_{1}]$ be maximal such that $m_{\eta}\leq
h_{\eta}(t)\leq m_{\eta}+1$ for any $t\in\lbrack t_{0},t_{1}]$ (see figure
\ref{Figure1}). Choose $r\geq1$ and $t_{0}=t^{(0)}<t^{(1)}<\cdots
<t^{(r)}=t_{1}$ satisfying the following conditions: for $1\leq a\leq r$

(1) either $h_{\eta}(t^{(a-1})=h_{\eta}(t^{(a)})$ and $h_{\eta}(t)\geq
h_{\eta}(t^{(a)})$ on $[t^{(a-1)},t^{(a)}]$,

(2) or $h_{\eta}$ is strictly decreasing on $[t^{(a-1)},t^{(a)}]$ and
$h_{\eta}(t)\geq h_{\eta}(t^{(a-1)})$ on $[0,t^{(a-1)}]$.

\noindent We set $t^{(-1)}=0$ and $t^{(r+1)}=1$ and, for $0\leq a\leq r+1$, we
denote by $\eta_{a}$ the elementary path defined by
\[
\forall u\in\lbrack0,1]\quad\eta_{a}(u)=\eta(t^{(a-1)}+u(t^{(a)}%
-t^{(a-1)}))-\eta(t^{(a-1)}).
\]
Observe that $\eta_{a}$ is the elementary path whose image translated by
$\eta(t^{(a-1)})$ coincides with the restriction of $\eta$ on $[t^{(a-1)}%
,t^{(a)}]$; the path $\eta$ decomposes as follows
\[
\eta=\eta_{0}\ast\eta_{1}\ast\cdots\ast\eta_{r}\ast\eta_{r+1}.
\]
For $1\leq a\leq r+1$, we also set $\eta_{a}^{\prime}=\eta_{a}$ in case (1)
and $\eta_{a}^{\prime}=s_{\alpha_{i}}(\eta_{a})$ in case (2). For
$i\in\{1,\cdots,n\}$, we set
\[
\tilde{e}_{i}(\eta)=\left\{
\begin{array}
[c]{ll}%
{\mathbf{0}} & {\text{if }}\ h_{\eta}(1)<m_{\eta}+1,\\
\eta_{0}\ast\eta_{1}^{\prime}\ast\cdots\ast\eta_{r}^{\prime}\ast\eta_{r+1} &
{\text{otherwise. }}%
\end{array}
\right.
\]
To define the $\tilde{f}_{i}$, we first propose another decomposition of the
path $\eta$. If $h_{\eta}(1)<m_{\eta}+1$, then $\tilde{f}_{i}(\eta
)={\mathbf{0}}$. Otherwise ($h_{\eta}(1)\geq m_{\eta}+1$), set $t_{0}^{\prime
}:=\sup\{t\in\lbrack0,1]\mid h_{\eta}(t_{0}^{\prime})=m_{\eta}\}$ and let
$t_{1}^{\prime}\in\lbrack t_{0}^{\prime},1]$ be minimal such that $h_{\eta
}(t)\geq m_{\eta}+1$ for $t\in\lbrack t_{1}^{\prime},1]$ (see figure
\ref{Figure1}). Choose $r\geq1$ and $t_{0}^{\prime}=t^{(0)}<t^{(1)}%
<\cdots<t^{(r)}=t_{1}^{\prime}$ satisfying the following conditions: for
$1\leq a\leq r$

(3) either $h_{\eta}(t^{(a-1})=h_{\eta}(t^{(a)})$ and $h_{\eta}(t)\geq
h_{\eta}(t^{(a-1)})$ on $[t^{(a-1)},t^{(a)}],$

(4) or $h_{\eta}$ is strictly increasing on $[t^{(a-1)},t^{(a)}]$ and
$h_{\eta}(t)\geq h_{\eta}(t^{(a)})$ on $[t^{(a)},1].$

\noindent We set $t^{(-1)}=0$ and $t^{(r+1)}=1$ and, for $0\leq a\leq r+1$, we
denote by $\eta_{a}$ the elementary path defined by
\[
\forall u\in\lbrack0,1]\quad\eta_{a}(u)=\eta(t^{(a-1)}+u(t^{(a)}%
-t^{(a-1)}))-\eta(t^{(a-1)}).
\]
As above, the path $\eta$ decomposes as $\eta=\eta_{0}\ast\eta_{1}\ast
\cdots\ast\eta_{r}\ast\eta_{r+1}$; for $1\leq a\leq r+1$, we thus set
$\eta_{a}^{\prime}=\eta_{a}$ in case (3) and $\eta_{a}^{\prime}=s_{\alpha_{i}%
}(\eta_{a})$ in case (4) and the operator $\tilde{f}_{i},1\leq i\leq n,$ is
defined by
\[
\tilde{f}_{i}(\eta)=\left\{
\begin{array}
[c]{ll}%
{\mathbf{0}} & {\text{if }}\ h_{\eta}(1)<m_{\eta}+1,\\
\eta_{0}\ast\eta_{1}^{\prime}\ast\cdots\ast\eta_{r}^{\prime}\ast\eta_{r+1} &
{\mathrm{otherwise.}}%
\end{array}
\right.
\]
%

\begin{figure}
[ptb]
\begin{center}
\fbox{\includegraphics[
trim=0.000000in 0.000000in -0.029974in -0.002869in,
height=10.0342cm,
width=12.1166cm
]%
{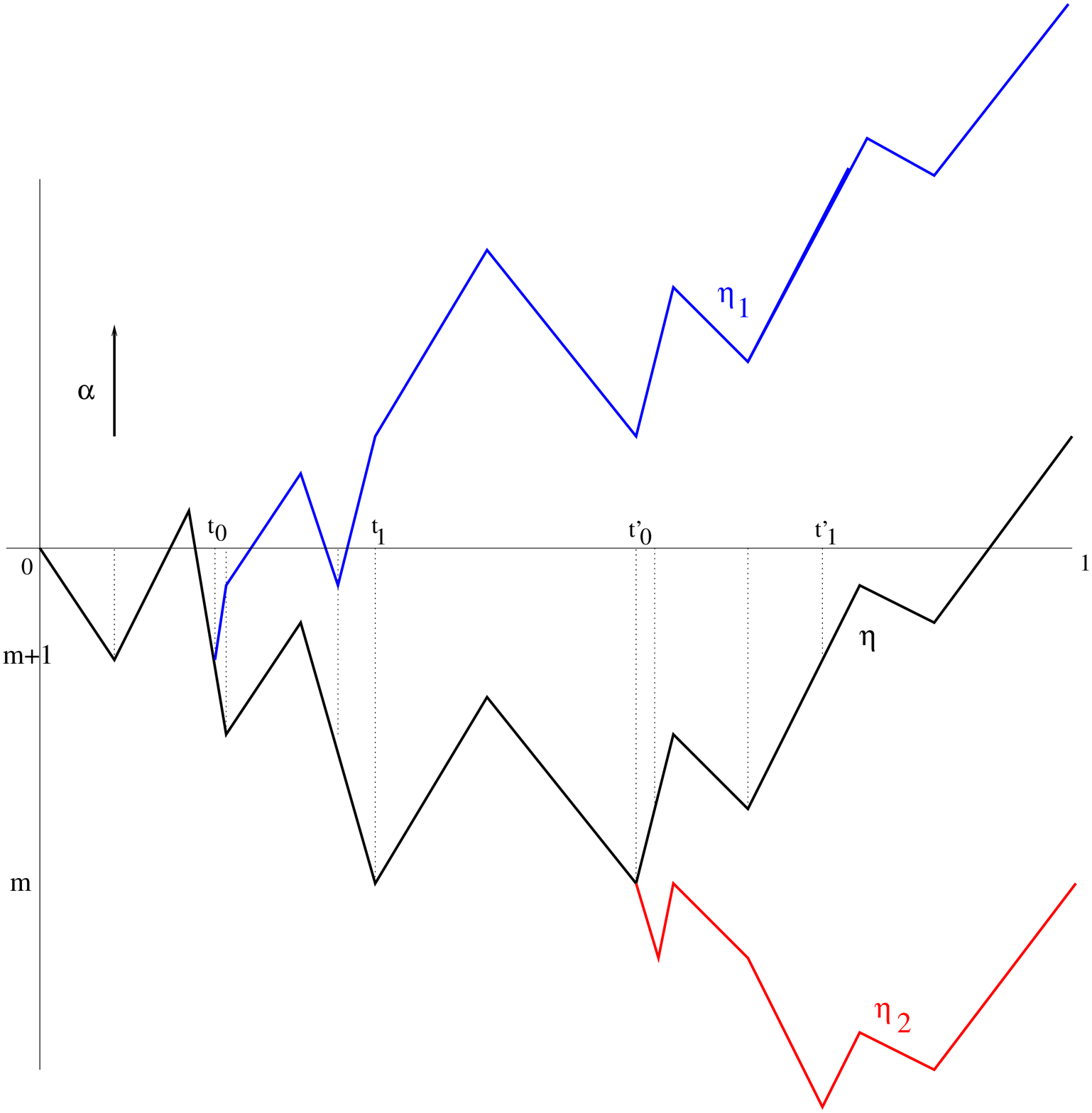}%
}\caption{Paths $\eta$, $\eta_{1}=\tilde{e}_{i}(\eta)$ and $\eta_{2}=\tilde
{f}_{i}(\eta)$}%
\label{Figure1}%
\end{center}
\end{figure}

\noindent\textbf{Remarks:} {1. When ${\mathfrak{g}}$ is finite-dimensional,
the symmetric bilinear form $\langle\cdot,\cdot\rangle$ can be assumed
positive so that elements of $W$ are isometries. The paths $\eta,\tilde{e}%
_{i}(\eta)$ and $\tilde{f}_{i}(\eta)$ have the same length. This is no longer
true when ${\mathfrak{g}}$ is of affine or indefinite type.}

2. When $\tilde{e}_{i}(\eta)$ is computed, the segments of $\eta$ which are
replaced by their symmetric under $s_{\alpha_{i}}$ correspond to intervals
where $h_{\eta}$ is strictly decreasing. This implies that $h_{\eta}(t)\leq
h_{\tilde{e}_{i}(\eta)}(t)$ for any $t\in\lbrack0,1]$. Similarly, we have
$h_{\eta}(t)\geq h_{\tilde{f}_{i}(\eta)}(t)$ for any $t\in\lbrack0,1]$.
\bigskip

\noindent The operators $\tilde{e}_{i}$ and $\tilde{f}_{i}$ satisfy the
following properties :

\begin{proposition}
\ \label{prop-ac_etil}

\begin{enumerate}
\item Assume $\tilde{e}_{i}(\eta)\neq{\mathbf{0}};\;$then $\tilde{e}_{i}%
(\eta)(1)=\eta(1)+\alpha_{i}$ and $\tilde{f}_{i}(\tilde{e}_{i}(\eta))=\eta$.

\item Assume $\tilde{f}_{i}(\eta)\neq{\mathbf{0}};\;$then $\tilde{f}_{i}%
(\eta)(1)=\eta(1)-\alpha_{i}$ and $\tilde{e}_{i}(\tilde{f}_{i}(\eta))=\eta$.

\item A path $\eta\in{\mathcal{P}}$ satisfies $\tilde{e}_{i}(\eta
)={\mathbf{0}}$ for any $i=1,\ldots,n$ if and only if $\operatorname{Im}%
\eta+\rho$ is contained in $\mathring{{\mathcal{C}}}$.
\end{enumerate}
\end{proposition}

{\noindent\textbf{Remark:} It also directly follows from the definition of
$\tilde{f}_{i}(\eta)$ that there exists a piecewise linear increasing map $g$
defined on $[0,1]$ satisfying}
\begin{equation}
\eta(t)-\tilde{f}_{i}(\eta)(t)=g(t)\alpha_{i}\text{ for any }t\in
\lbrack0,1]\label{func_g}%
\end{equation}
and $g(0)=0,$ $g(1)=1$.

\bigskip

We may endow ${\mathcal{P}}$ with the structure of a Kashiwara crystal: this
means that ${\mathcal{P}}$ has the structure of a colored oriented graph by
drawing an arrow $\eta\overset{i}{\rightarrow}\eta^{\prime}$ between the two
paths $\eta,\eta^{\prime}$ of ${\mathcal{P}}$ as soon as $\tilde{f}_{i}%
(\eta)=\eta^{\prime}$ (or equivalently $\eta=\tilde{e}_{i}(\eta^{\prime}%
)$).\ For any $\eta\in{\mathcal{P}}$, we denote by $B(\eta)$ the connected
component of $\eta$ i.e. the subgraph of ${\mathcal{P}}$ obtained by applying
operators $\tilde{e}_{i}$ and $\tilde{f}_{i}$, $i=1,\ldots,n$ to $\eta
$.\ \bigskip

\noindent For any path $\eta\in{\mathcal{P}}$ and $i=1,\ldots,n$, set
$\varepsilon_{i}(\eta)=\max\{k\in{\mathbb{Z}}_{\geq0}\mid\tilde{e}_{i}%
^{k}(\eta)={\mathbf{0}}\}$ and $\varphi_{i}(\eta)=\max\{k\in{\mathbb{Z}}%
_{\geq0}\mid\tilde{f}_{i}^{k}(\eta)={\mathbf{0}}\}$; one easily checks that
$\varepsilon_{i}(\eta)$ and $\varphi_{i}(\eta)$ are finite. \bigskip

\noindent We now introduce the following notations

$\bullet\quad{\mathcal{P}}_{\min{\mathbb{Z}}}$ is the set of \textit{integral
paths}, that is paths $\eta$ such that $m_{\eta}=\min_{t\in\lbrack
0,1]}\{\langle\eta(t),\alpha_{i}^{\vee}\rangle\}$ belongs to ${\mathbb{Z}}$
for any $i=1,\ldots,n$.

$\bullet\quad{\mathcal{C}}$ is the cone in ${\mathfrak{h}}_{{\mathbb{R}}}^{*}$
defined by $\displaystyle {\mathcal{C}}=\{x\in{\mathfrak{h}}_{{\mathbb{R}}%
}^{\ast}\mid x(h_{i})\geq0\}.$

$\bullet\quad\mathring{{\mathcal{C}}}$ is the interior of ${{\mathcal{C}}}$;
it is defined by $\displaystyle \mathring{{\mathcal{C}}}=\{x\in{\mathfrak{h}%
}_{{\mathbb{R}}}^{\ast}\mid x(h_{i})>0\}{\text{.}}$

\noindent One gets the

\begin{proposition}
\ \label{Prop_HP} Let $\eta$ and $\pi$ two paths in ${\mathcal{P}}%
_{\min{\mathbb{Z}}}$. Then

\begin{enumerate}
\item the concatenation $\pi\ast\eta$ belongs to ${\mathcal{P}}_{\min
{\mathbb{Z}}}$,

\item for any $i=1,\ldots,n$ we have
\begin{equation}
\tilde{e}_{i}(\eta\ast\pi)=\left\{
\begin{array}
[c]{ll}%
\eta\ast\tilde{e}_{i}(\pi) & {\text{if }}\varepsilon_{i}(\pi)>\varphi_{i}%
(\eta)\\
\tilde{e}_{i}(\eta)\ast\pi & {\text{otherwise,}}%
\end{array}
\right.  {\text{and }}\tilde{f}_{i}(\eta\ast\pi)=\left\{
\begin{array}
[c]{ll}%
\tilde{f}_{i}(\eta)\ast\pi & {\text{if }}\varphi_{i}(\eta)>\varepsilon_{i}%
(\pi)\\
\eta\ast\tilde{f}_{i}(\pi) & {\text{otherwise.}}%
\end{array}
\right.  \label{Ten_Prod}%
\end{equation}
In particular, $\tilde{e}_{i}(\eta\ast\pi)={\mathbf{0}}$ if and only if
$\tilde{e}_{i}(\eta)={\mathbf{0}}$ and $\varepsilon_{i}(\pi)\leq\varphi
_{i}(\eta)$ for any $i=1,\ldots,n$.

\item $\tilde{e}_{i}(\eta)={\mathbf{0}}$ for any $i=1,\ldots,n$ if and only if
$\operatorname{Im}\eta$ is contained in ${\mathcal{C}}$.
\end{enumerate}
\end{proposition}

The following theorem summarizes crucial results of Littelmann (see
\cite{Lit1}, \cite{Lit2} and \cite{Lit3}).\ 

\begin{theorem}
\label{Th_Littel}Consider $\lambda,\mu$ and $\kappa$ dominant weights and
choose arbitrarily elementary paths $\eta_{\lambda},\eta_{\mu}$ and
$\eta_{\kappa}$ in ${\mathcal{P}}$ such that $\operatorname{Im}\eta_{\lambda
}\subset{\mathcal{C}}$, $\operatorname{Im}\eta_{\mu}\subset{\mathcal{C}}$ and
$\operatorname{Im}\eta_{\kappa}\subset{\mathcal{C}}$ and joining respectively
$0$ to $\lambda$, $0$ to $\mu$ and $0$ to $\kappa$.

\begin{enumerate}
\item We have $B(\eta_{\lambda}):=\{\tilde{f}_{i_{1}}\cdots\tilde{f}_{i_{k}%
}\eta_{\lambda}\mid k\in{\mathbb{N}}{\text{ and }}1\leq i_{1},\cdots,i_{k}\leq
n\}\setminus\{{\mathbf{0}}\}.$

In particular ${\mathrm{wt}}(\eta)-{\mathrm{wt}}(\eta_{\lambda})\in Q_{+}$ for
any $\eta\in B(\eta_{\lambda})$.

\item The graph $B(\eta_{\lambda})$ is contained in ${\mathcal{P}}%
_{\min{\mathbb{Z}}}.$

\item If $\eta_{\lambda}^{\prime}$ is another elementary path from $0$ to
$\lambda$ such that $\operatorname{Im}\eta_{\lambda}^{\prime}$ is contained in
${\mathcal{C}}$, then $B(\eta_{\lambda})$ and $B(\eta_{\lambda}^{\prime})$ are
isomorphic as oriented graphs i.e. there exists a bijection $\theta
:B(\eta_{\lambda})\rightarrow B(\eta_{\lambda}^{\prime})$ which commutes with
the action of the operators $\tilde{e}_{i}$ and $\tilde{f}_{i}$,
$i=1,\ldots,n$.

\item The crystal $B(\eta_{\lambda})$ is isomorphic to the Kashiwara crystal
graph $B(\lambda)$ associated to the $U_{q}({\mathfrak{g}})$-module of highest
weight $\lambda$.

\item We have
\begin{equation}
\label{slambda}s_{\lambda}=\sum_{\eta\in B(\eta_{\lambda})}e^{\eta(1)}.
\end{equation}

\item For any $i=1,\ldots,n$ and any $b\in B(\eta_{\lambda})$, let $s_{i}(b)$
be the unique path in $B(\eta_{\lambda})$ such that
\[
\varphi_{i}(s_{i}(b))=\varepsilon_{i}(b){\text{ and }}\varepsilon_{i}%
(s_{i}(b))=\varphi_{i}(b)
\]
(in other words, $s_{i}$ acts on each $i$-chain ${\mathcal{C}}_{i}$ as the
symmetry with respect to the center of ${\mathcal{C}}_{i}$). The actions of
the $s_{i}$'s extend to an action of ${\mathsf{W}}$ on ${\mathcal{P}}$ which
stabilizes $B(\eta_{\lambda})$. In particular, for any $w\in{\mathsf{W}}$ and
any $b\in B(\eta_{\lambda})$, we have $w(b)\in B(\eta_{\lambda})$ and
${\mathrm{wt}}(w(b))=w({\mathrm{wt}}(b))$.\footnote{This action should not be
confused with that defined in (\ref{actW}) which does not stabilize
$B(\eta_{\lambda})$ in general.}

\item For any $b\in B(\eta_{\lambda})$ we have ${\mathrm{wt}}(b)=\sum
_{i=1}^{n}(\varphi_{i}(b)-\varepsilon_{i}(b))\omega_{i}.$

\item Given any integer $\ell\geq0$, set
\begin{equation}
B(\eta_{\mu})\ast B(\eta_{\kappa})^{\ast\ell}=\{\pi=\eta\ast\eta_{1}\ast
\cdots\ast\eta_{\ell}\in{\mathcal{P}}\mid\eta\in B(\eta_{\mu}){\text{ and }%
}\eta_{k}\in B(\eta_{\kappa})\ {\text{ for any }}k=1,\ldots,\ell\}.
\label{def_crysta_*}%
\end{equation}
The graph $B(\eta_{\mu})\ast B(\eta_{\kappa})^{\ast\ell}$ is contained in
${\mathcal{P}}_{\min{\mathbb{Z}}}.$

\item The multiplicity $m_{\mu,\kappa}^{\lambda}$ defined in (\ref{def_multi})
is equal to the number of paths of the form $\mu\ast\eta$ with $\eta\in
B(\eta_{\kappa})$ contained in ${\mathcal{C}}$.

\item The multiplicity $f_{\lambda/\mu}^{\ell}$ defined in (\ref{def_multi})
is equal to cardinality of the set
\[
H_{\lambda/\mu}^{\ell}:=\{\pi\in B(\eta_{\mu})\ast B(\eta_{\kappa})^{\ast\ell
}\mid\tilde{e}_{i}(\pi)=0\ {\text{ for any }}i=1,\ldots,n{\text{ and }}%
\pi(1)=\lambda\}.
\]
Each path $\pi=\eta\ast\eta_{1}\ast\cdots\ast\eta_{\ell}\in H_{\lambda/\mu
}^{\ell}$ verifies $\operatorname{Im}\pi\subset{\mathcal{C}}$ and $\eta
=\eta_{\mu}$.
\end{enumerate}
\end{theorem}

\noindent\textbf{Remarks:}\ 1. Combining assertion (2) of Proposition
\ref{prop-ac_etil} together with assertions (1) and (5) of the Theorem
\ref{Th_Littel}, one may check that the function $e^{-\lambda}s_{\lambda}$ is
in fact a polynomial in the variables $T_{i}=e^{-\alpha_{i}}$, namely
\begin{equation}
s_{\lambda}=e^{\lambda}S_{\lambda}(T_{1},\ldots,T_{n}) \label{defS}%
\end{equation}
where $S_{\lambda}\in{\mathbb{C}}[X_{1},\ldots,X_{n}]$. Observe also that the
quantity $S_{\infty}:=\prod_{\alpha\in R_{+}}\frac{1}{(1-e^{-\alpha
})^{m_{\alpha}}}$ is a formal power series in the variables $T_{1}%
,\ldots,T_{n}$.\ M. Kashiwara proved (see for instance \cite{Joseph}
\S \ 20.7) that the crystal $B(\lambda)$ admits a projective limit $B(\infty)$
when $\lambda$ tends to infinity and that
\[
{\mathrm{char}}(B(\infty))=\sum_{b\in B(\infty)}e^{{\mathrm{wt}}(b)}%
=S_{\infty}.
\]
Now, since $B(\lambda)$ can be embedded in $B(\infty)$ up to a translation by
the weights by $\lambda$, we have
\begin{equation}
S_{\lambda}(T_{1},\ldots,T_{n})\leq S_{\infty}(T_{1},\ldots,T_{n});
\label{inequality}%
\end{equation}
in other words $S_{\infty}(T_{1},\ldots,T_{n})=S_{\lambda}(T_{1},\ldots
,T_{n})+\sum_{\mu\in Q_{+}}a_{\mu}T^{\mu}$ where the coefficients $a_{\mu}$
are nonnegative integers.

2. Using assertion (1) of Theorem \ref{Th_Littel}, we obtain $m_{\mu,\delta
}^{\lambda}\neq0$ only if $\mu+\delta-\lambda\in Q_{+}$. Similarly, when
$f_{\lambda/\mu}^{\delta,\ell}\neq0$ one necessarily has $\mu+\ell
\delta-\lambda\in Q_{+}$.

3. A minuscule weight is a dominant weight $\kappa\in P_{+}$ such that the
weights of $V(\kappa)$ are exactly those of the orbit ${\mathsf{W}}\cdot
\kappa$. In this case, if we take $\eta_{\kappa}:t\mapsto t\kappa$, the
crystal $B(\eta_{\kappa})$ contains only the paths $\eta:t\mapsto tw(\kappa)$.
In particular, these paths are lines.

4: Given any path $\eta_{\lambda}$ such that $\operatorname{Im}\eta_{\lambda
}\subset{\mathcal{C}}$, the set of paths $B(\eta_{\lambda})$ is in general
very difficult to describe (even in the finite type cases). Nevertheless, for
the classical types or type $G_{2}$ and a particular choice of $\eta_{\lambda
}$, the sets $B(\eta_{\lambda})$ can be made explicit by using generalizations
of semistandard tableaux (see for example \cite{Lec} and the references therein).

\bigskip{The height $ht(\eta)$ of a path $\eta\in B(\eta_{\lambda})$ is the
length of any path in $B(\eta_{\lambda})$ from $\eta_{\lambda}$ to $\eta$. For
any $a\geq0$, we denote by $B(\eta_{\lambda})_{a}$ the set of paths in
$B(\eta_{\lambda})$ at height $a.$ Each subset $B(\eta_{\lambda})_{a}$ is
finite and we have%
\begin{equation}
B(\eta_{\lambda})={\displaystyle\bigsqcup\limits_{a\geq0}}B(\eta_{\lambda
})_{a}.\label{dec_depth}%
\end{equation}
By Proposition \ref{prop-ac_etil}, $ht(\eta)$ is equal to the number of simple
roots appearing in the decomposition of $\mathrm{wt}(\eta_{\lambda
})-\mathrm{wt}(\eta)$ on the basis $\{\alpha_{1},\ldots,\alpha_{n}\}$.}

\section{Random paths and symmetrizable Kac-Moody algebras}

\subsection{Probability distribution on elementary paths}

\label{Probability distribution on elementary paths}

Consider $\kappa\in P_{+}$ and a path $\pi_{\kappa}\in{\mathcal{P}}$ from $0$
to $\kappa$ such that $\operatorname{Im}\pi_{\kappa}$ is contained in
${\mathcal{C}}$. Let $B(\pi_{\kappa})$ be the connected component of
${\mathcal{P}}$ containing $\pi_{\kappa}$.\ We now endow $B(\pi_{\kappa})$
with a probability distribution $p_{\kappa}$, which will be characterized by
the datum of a $n$-tuple $\tau=(\tau_{1},\ldots,\tau_{n})\in{\mathbb{R}}%
_{>0}^{n}$ (each $\tau_{i}$ can be regarded as attached to the positive simple
root $\alpha_{i}$). For any $u=u_{1}\alpha_{1}+\cdots+u_{n}\alpha_{n}\in Q$,
we set $\tau^{u}=\tau_{1}^{u_{1}}\cdots\tau_{n}^{u_{n}}$. Let $\pi\in
B(\pi_{\kappa})$: by assertion (1) of Theorem \ref{Th_Littel}, one gets
\[
\pi(1)={\mathrm{wt}}(\pi)=\kappa-\sum_{i=1}^{n}u_{i}(\pi)\alpha_{i}%
\]
where $u_{i}(\pi)\in{\mathbb{N}}$ for any $i=1,\ldots,n$. We have $S_{\kappa
}(\tau):=S_{\kappa}(\tau_{1},\ldots,\tau_{n})=\sum_{\pi\in B(\pi_{\kappa}%
)}\tau^{\kappa-{\mathrm{wt}}(\pi)}.$

\begin{proposition}
\label{Prop_Conv} For any $\kappa\in P_{+}$,

\begin{enumerate}
\item if $A$ is of finite type then $0<S_{\kappa}(\tau)<\infty$ for any
$\tau\in{\mathbb{R}}_{>0}^{n},$

\item if $A$ is of affine type then $0<S_{\kappa}(\tau)<\infty$ for any
$\tau\in]0,1[^{n}$,

\item if $A$ is of indefinite type then $0<S_{\kappa}(\tau)<\infty$ for any
$\tau\in]0,\frac{1}{n}[^{n}$.
\end{enumerate}
\end{proposition}

\begin{proof}
\noindent The inequality $S_{\kappa}(\tau)>0$ is immediate since $\tau_{i}>0$
for any $i=1,\ldots,n$. When $A$ is of finite type, the crystal $B(\pi
_{\kappa})$ is finite, so that $S_{\kappa}(\tau)<\infty$. When $A$ is not of
finite type, let $\bar{\tau}=\max(\tau_{i},i=1,\ldots,n)$.\ We have by
(\ref{inequality})%

\[
S_{\kappa}(\tau)\leq S_{\infty}(\tau)=\prod_{\alpha\in R_{+}}\frac{1}%
{(1-\tau^{\alpha})^{m_{\alpha}}}\leq\prod_{\alpha\in R_{+}}\frac{1}%
{(1-\bar{\tau}^{ht(\alpha)})^{m_{\alpha}}}%
\]
and it suffices to prove that
\begin{equation}
S_{\infty}^{\ast}(\bar{\tau})=S_{\infty}(\bar{\tau},\ldots,\bar{\tau}%
)=\prod_{\alpha\in R_{+}}\frac{1}{(1-\bar{\tau}^{ht(\alpha)})^{m_{\alpha}}%
}<+\infty.\label{S*}%
\end{equation}

$\bullet$ Assume first that $A$ is of affine type different from $A_{2n}%
^{(2)}$. By (\ref{PosAffRoots}) and (\ref{multiaffine}), we have
\[
\prod_{\alpha\in R_{+}}\frac{1}{(1-\bar{\tau}^{ht(\alpha)})^{m_{\alpha}}}%
\leq\left(  \prod_{\alpha\in R_{+}^{\circ}}\frac{1}{1-\bar{\tau}^{ht(\alpha)}%
}\right)  \left(  \prod_{k=1}^{+\infty}\frac{1}{(1-\bar{\tau}^{kht(\delta
)})^{n}}\right)  \prod_{\alpha\in R^{\circ}}\left(  \prod_{k=1}^{+\infty}%
\frac{1}{1-\bar{\tau}^{ht(\alpha+kr\delta)}}\right)
\]
since $0<\bar{\tau}<1$ for any $\alpha\in R_{+}$ and $R^{\circ}$ is finite. We
have to prove that the infinite products in the above expression are finite.
Since $ht(\delta)\geq n$, we have $\bar{\tau}^{h(\delta)}\leq\bar{\tau}^{n}$;
moreover $\alpha+kr\delta\in Q_{+}$ for any $k\geq1$ and $\alpha\in R^{\circ}%
$.\ We therefore get
\[
\prod_{k=1}^{+\infty}\frac{1}{(1-\bar{\tau}^{kht(\delta)})^{n}}\leq\left(
\prod_{k=1}^{+\infty}\frac{1}{1-\bar{\tau}^{kn}}\right)  ^{n}<+\infty
\]
since the series $\sum_{k=1}^{+\infty}\ln(1-\bar{\tau}^{kn})$ converges for
$\bar{\tau}^{n}\in]0,1[$. Similarly, since $\bar{\tau}^{rn}\in]0,1[$ one gets
\[
\prod_{k=1}^{+\infty}\frac{1}{1-\bar{\tau}^{ht(\alpha+kr\delta)}}\leq
\prod_{k=1}^{+\infty}\frac{1}{1-\bar{\tau}^{ht(\alpha)}\bar{\tau}^{krn}%
}<+\infty.
\]
The case $A_{2n}^{(2)}$ is obtained by the same arguments.

$\bullet$ Secondly, assume that $A$ is of indefinite type. By (\ref{mag_mult}%
), we have $\displaystyle S_{\infty}(\bar{\tau})\leq\prod_{\alpha\in R_{+}%
}\left(  {\frac{1}{1-\bar{\tau}^{ht(\alpha)}}}\right)  ^{n^{ht(\alpha)}}%
$.\ Moreover, since $0<\bar{\tau}<1$ for any $\beta\in Q_{+}$ and
$R_{+}\subset Q_{+}$, we have also
\[
S_{\infty}^{\ast}(\bar{\tau})\leq\prod_{\overset{\beta\in Q_{+}}{\beta\neq0}%
}\frac{1}{(1-\bar{\tau}^{ht(\beta)})^{n^{ht(\beta)}}}=\prod_{k=1}^{+\infty
}\prod_{\overset{\beta\in Q_{+}}{ht(\beta)=k}}\frac{1}{(1-\bar{\tau}%
^{k})^{n^{k}}}%
\]
with
\[
\prod_{\overset{\beta\in Q_{+}}{ht(\beta)=k}}\frac{1}{(1-\bar{\tau}%
^{k})^{n^{k}}}\leq\left(  {\frac{1}{1-\bar{\tau}^{k}}}\right)  ^{(k+1)^{n}%
n^{k}}%
\]
since ${\mathrm{card}}(\{\beta\in Q_{+}\mid ht(\beta)=k\})\leq(k+1)^{n}$. We
thus get
\[
S_{\infty}^{\ast}(\bar{\tau})\leq\prod_{k=1}^{+\infty}\frac{1}{(1-\bar{\tau
}^{k})^{n^{k}(k+1)^{n}}}<+\infty
\]
using the fact that the series $\sum_{k=1}^{+\infty}n^{k}(k+1)^{n}\ln
(1-\bar{\tau}^{k})$ converges for $\bar{\tau}\in]0,\frac{1}{n}[$.
\end{proof}

\bigskip

{The previous proposition has three important corollaries. First set
$T_{\kappa}(\tau):=T_{\kappa}(\tau_{1},\ldots,\tau_{n})=\sum_{\pi\in
B(\pi_{\kappa})}ht(\pi)\tau^{\kappa-{\mathrm{wt}}(\pi)}$.}

\begin{corollary}
\label{Cor_T}For any $\kappa\in P_{+}$,

\begin{enumerate}
\item if $A$ is of finite type then $0<T_{\kappa}(\tau)<\infty$ for any
$\tau\in{\mathbb{R}}_{>0}^{n},$

\item if $A$ is of affine type then $0<T_{\kappa}(\tau)<\infty$ for any
$\tau\in]0,1[^{n}$,

\item if $A$ is of indefinite type then $0<T_{\kappa}(\tau)<\infty$ for any
$\tau\in]0,\frac{1}{n}[^{n}$.
\end{enumerate}
\end{corollary}

\begin{proof}
This is clear when $A$ is of finite type. For assertion 2 and 3, let
$\bar{\tau}=\max(\tau_{i},i=1,\ldots,n)$.\ In the previous proof, we have
established that $S_{\infty}^{\ast}(\bar{\tau})$ is finite. Set $S_{\kappa
}^{\ast}(\bar{\tau})=S_{\kappa}(\bar{\tau},\ldots,\bar{\tau})$. Since
$S_{\kappa}^{\ast}(\bar{\tau})\leq S_{\infty}^{\ast}(\bar{\tau})$, the series
$S_{\kappa}^{\ast}(\bar{\tau})$ is also finite. This means that for any
$\bar{\tau}\in]0,1[$ we have%
\[
S_{\kappa}^{\ast}(\bar{\tau})=\sum_{\pi\in B(\pi_{\kappa})}\bar{\tau}%
^{ht(\pi)}=\sum_{a\geq0}m(a)\bar{\tau}^{a}<+\infty
\]
where $m(a)$ is the number of paths in $B(\pi_{\kappa})_{a}$ (see
(\ref{dec_depth})). It follows that $T_{\kappa}^{\ast}(\bar{\tau})=\sum
_{a\geq0}am(a)\bar{\tau}^{a}$ is also finite for any $\bar{\tau}\in]0,1[$. Now
we have
\[
T_{\kappa}(\tau)\leq T_{\kappa}(\bar{\tau},\ldots,\bar{\tau})=\sum_{\pi\in
B(\pi_{\kappa})}ht(\pi)\bar{\tau}^{ht(\pi)}=T_{\kappa}^{\ast}(\bar{\tau
})<+\infty.
\]

\end{proof}

\bigskip

\textbf{From now on}, we write ${\mathcal{T}}$ for the set of $n$-tuples
$\tau=(\tau_{1},\cdots,\tau_{n})\in{\mathbb{R}}_{>0}^{n}$ such that

$\bullet\quad\tau_{i}\in]0,1[$ for $1\leq i\leq n$ when $A$ is of finite or
affine type,

$\bullet\quad\tau_{i}\in]0,{\frac{1}{n}}[$ for $1\leq i\leq n$ when $A$ is of
indefinite type.

\begin{corollary}
\label{Cor_Cvm} For any $\mu\in P_{+}$ and $w\in{\mathsf{W}}$, the weight
$\mu+\rho-w(\mu+\rho)$ belongs to $Q_{+};$ moreover, for $\tau\in{\mathcal{T}%
}$, one gets
\[
\left\vert \sum_{w\in{\mathsf{W}}}\varepsilon(w)\tau^{\mu+\rho-w(\mu+\rho
)}\right\vert \leq\sum_{w\in{\mathsf{W}}}\tau^{\mu+\rho-w(\mu+\rho)}<+\infty.
\]

\end{corollary}

\begin{proof}
By the Weyl-Kac character formula, one gets
\[
e^{-\mu}s_{\mu}=\frac{\sum_{w\in{\mathsf{W}}}\varepsilon(w)e^{w(\mu+\rho
)-\rho-\mu}}{\prod_{\alpha\in R_{+}}(1-e^{-\alpha})^{m_{\alpha}}}.
\]
Since $e^{-\mu}s_{\mu}$ and $\prod_{\alpha\in R_{+}}(1-e^{-\alpha}%
)^{m_{\alpha}}$ are polynomial in $e^{-\beta}$ with $\beta\in Q_{+}$, we have
$\mu+\rho-w(\mu+\rho)\in Q_{+}$ for any $w\in{\mathsf{W}}$.\ Now observe that
$\mu+\rho$ belongs to $P_{+}$ and is the dominant weight of $V(\mu+\rho)$,
each $w(\mu+\rho)$ is thus also a weight of $V(\mu+\rho)$. Therefore, the
coefficients of the decomposition of $s_{\mu+\rho}-\sum_{w\in{\mathsf{W}}%
}e^{w(\mu+\rho)}$ on the basis $\{e^{\beta}\mid\beta\in P\}$ are nonnegative;
in other words $\displaystyle \sum_{w\in{\mathsf{W}}}e^{w(\mu+\rho)}\leq
s_{\mu+\rho}$ which readily implies that $\sum_{w\in{\mathsf{W}}}e^{w(\mu
+\rho)-\mu-\rho}\leq e^{-\mu-\rho}s_{\mu+\rho}$. By specializing
$e^{-\alpha_{i}}=\tau_{i}$, one gets $\sum_{w\in{\mathsf{W}}}\tau^{\mu
+\rho-w(\mu+\rho)}\leq S_{\mu+\rho}(\tau)<+\infty$.
\end{proof}

\begin{definition}
We define the probability distribution $p$ on $B(\pi_{\kappa})$setting
$\displaystyle p_{\pi}=\frac{\tau^{\kappa-{\mathrm{wt}}(\pi)}}{S_{\kappa}%
(\tau)}.$
\end{definition}

\noindent\textbf{Remark:} By Assertion 3 of Theorem \ref{Th_Littel}, for
$\pi_{\kappa}^{\prime}$ another elementary path from $0$ to $\kappa$ such that
$\operatorname{Im}\pi_{\kappa}^{\prime}$ is contained in ${\mathcal{C}}$,
there exists an isomorphism $\Theta$ between the crystals $B(\pi_{\kappa})$
and $B(\pi_{\kappa}^{\prime})$ and one gets $p_{\pi}=p_{\Theta(\pi)}$ for any
$\pi\in B(\pi_{\kappa})$. Therefore, the probability distributions we use on
the graph $B(\pi_{\kappa})$ are invariant by crystal isomorphisms. \bigskip

Let $X$ a random variable with values in $B(\pi_{\kappa})$ and probability
distribution $p$; as a direct consequence of Proposition \ref{Prop_Conv}, we
get the

\begin{corollary}
\label{Cor_mexists}The variable $X$ admits a moment of order $1$. Moreover the
series of functions
\[
m=\sum_{\pi\in B(\pi_{\kappa})}p_{\pi}\pi
\]
converges uniformly on $[0,1]$.
\end{corollary}

\begin{proof}
We can decompose $B(\pi_{\kappa})=\bigsqcup\limits_{a\geq0}B(\pi_{\kappa}%
)_{a}$ as in (\ref{dec_depth}). Then, we get for any $t\in\lbrack0,1]$%
\[
m(t)=\sum_{a\geq0}\sum_{\pi\in B(\pi_{\kappa})_{a}}p_{\pi}\pi(t).
\]
Consider $\pi\in B(\pi_{\kappa})_{a}$ and set $\pi=\tilde{f}_{i_{1}}%
\cdots\tilde{f}_{i_{a}}(\pi_{\kappa})$. By (\ref{func_g}) and an immediate
induction, there exist increasing piecewise linear maps $g_{1},\ldots,g_{a}$
from $[0,1]$ to itself with $g_{k}(0)=0$ and $g_{k}(1)=1$ for any
$k=1,\ldots,a$ such that
\[
\pi(t)=\pi_{\kappa}(t)-(g_{1}(t)\alpha_{i_{1}}+\cdots+g_{a}(t)\alpha_{i_{a}}).
\]
In particular $\left\Vert \pi(t)-\pi_{\kappa}(t)\right\Vert \leq\left\Vert
\alpha_{1}\right\Vert +\cdots+\left\Vert \alpha_{a}\right\Vert \leq Ca$ where
$C=\max_{\alpha\in\pi}\left\Vert \alpha\right\Vert $ is the norm of the
longest simple root. We thus get%
\[
\left\Vert \pi(t)\right\Vert \leq\left\Vert \pi_{\kappa}(t)\right\Vert
+\left\Vert \pi(t)-\pi_{\kappa}(t)\right\Vert \leq M+Cht(\pi)
\]
where $M=\max_{t\in\lbrack0,1]}\left\Vert \pi_{\kappa}(t)\right\Vert $. We
obtain $\max_{t\in\lbrack0,1]}\left\Vert p_{\pi}\pi(t)\right\Vert
\leq(M+Cht(\pi)\frac{\tau^{\kappa-{\mathrm{wt}}(\pi)}}{S_{\kappa}(\tau)}.$ But
the series
\[
S_{\kappa}(\tau)=\sum_{\pi\in B(\pi_{\kappa})}\tau^{\kappa-{\mathrm{wt}}(\pi
)}\text{ and }T_{\kappa}(\tau)=\sum_{\pi\in B(\pi_{\kappa})}ht(\pi
)\tau^{\kappa-{\mathrm{wt}}(\pi)}%
\]
converge by Proposition \ref{Prop_Conv} and Corollary \ref{Cor_T}. This means
that the series of functions $m$ converges uniformly on $[0,1]$.
\end{proof}

\subsection{Random paths of arbitrary length}

\label{Subsec-RPW}We now extend the notion of elementary random paths. Assume
that $\pi_{1},\ldots,\pi_{\ell}$ a family of elementary paths; the path
$\pi_{1}\otimes\cdots\otimes\pi_{\ell}$ of length $\ell$ is defined by: for
all $k\in\{1,\ldots,\ell-1\}$ and $t\in\lbrack k,k+1]$
\begin{equation}
\label{tens_path}\pi_{1}\otimes\cdots\otimes\pi_{\ell}(t)=\pi_{1}%
(1)+\cdots+\pi_{k}(1)+\pi_{k+1}(t-k).
\end{equation}
Let $B^{\otimes\ell}(\pi_{\kappa})$ be the set of paths of the form $b=\pi
_{1}\otimes\cdots\otimes\pi_{\ell}$ where $\pi_{1},\ldots,\pi_{\ell}$ are
elementary paths in $B(\pi_{\kappa})$; there exists a bijection $\Delta$
between $B^{\otimes\ell}(\pi_{\kappa})$ and the set $B^{\ast\ell}(\pi_{\kappa
})$ of paths in ${\mathcal{P}}$ obtained by concatenations of $\ell$ paths of
$B(\pi_{\kappa})$:
\begin{equation}
\Delta:\left\{
\begin{array}
[c]{clc}%
B^{\otimes\ell}(\pi_{\kappa}) & \longrightarrow & B^{\ast\ell}(\pi_{\kappa})\\
\pi_{1}\otimes\cdots\otimes\pi_{\ell} & \longmapsto & \pi_{1}\ast\cdots\ast
\pi_{\ell}%
\end{array}
\right.  . \label{DefDelta}%
\end{equation}
In fact $\pi_{1}\otimes\cdots\otimes\pi_{\ell}$ and $\pi_{1}\ast\cdots\ast
\pi_{\ell}$ coincide up to a reparametrization and we define the weight of
$b=\pi_{1}\otimes\cdots\otimes\pi_{\ell}$ setting
\[
{\mathrm{wt}}(b):={\mathrm{wt}}(\pi_{1})+\cdots+{\mathrm{wt}}(\pi_{\ell}%
)=\pi_{1}(1)+\cdots+\pi_{\ell}(1).
\]
We now endow $B^{\otimes\ell}(\pi_{\kappa})$ with the product probability
measure $p^{\otimes\ell}$ defined by
\begin{equation}
p^{\otimes\ell}(\pi_{1}\otimes\cdots\otimes\pi_{\ell})=p(\pi_{1})\cdots
p(\pi_{\ell})=\frac{\tau^{\ell\kappa-(\pi_{1}(1)+\cdots\pi_{\ell}(1))}%
}{S_{\kappa}(\tau)^{\ell}}=\frac{\tau^{\ell\kappa-{\mathrm{wt}}(b)}}%
{S_{\kappa}(\tau)^{\ell}}. \label{potimesell}%
\end{equation}
In particular, for any $b,b^{\prime}$ in $B^{\otimes\ell}(\pi_{\kappa})$ such
that ${\mathrm{wt}}(b)={\mathrm{wt}}(b^{\prime})$, one gets
\[
p^{\otimes l}(b)=p^{\otimes l}(b^{\prime}).
\]
Write $\Pi_{\ell}:B^{\otimes\ell}(\pi_{\kappa})\rightarrow B^{\otimes\ell
-1}(\pi_{\kappa})$ the projection defined by $\Pi_{\ell}(\pi_{1}\otimes
\cdots\otimes\pi_{\ell-1}\otimes\pi_{\ell})=\pi_{1}\otimes\cdots\otimes
\pi_{\ell-1}$; the sequence $(B^{\otimes\ell}(\pi_{\kappa}),\Pi_{\ell
},p^{\otimes\ell})_{\ell\geq1}$ is a projective system of probability spaces.
We denote by $(B^{\otimes{\mathbb{N}}}(\pi_{\kappa}),p^{\otimes{\mathbb{N}}})$
its injective limit; the elements of $B^{\otimes{\mathbb{N}}}(\pi_{\kappa})$
are infinite sequences $b=(\pi_{\ell})_{\ell\geq1}$ and by a slight abuse of
notation, we will also write $\Pi_{\ell}(b)=\pi_{1}\otimes\cdots\otimes
\pi_{\ell}$.\ 

Now let $X=(X_{\ell})_{\ell\geq1}$ a sequence of i.i.d. random variables with
values in $B(\pi_{\kappa})$ and probability distribution $p$; the random path
${\mathcal{W}}$ on $(B^{\otimes{\mathbb{N}}}(\pi_{\kappa}),p^{\otimes
{\mathbb{N}}})$ are thus defined by
\[
{\mathcal{W}}(t):=\Pi_{\ell}(X)(t)=X_{1}\otimes X_{2}\otimes\cdots\otimes
X_{\ell-1}\otimes X_{\ell}(t){\text{ for }}t\in\lbrack\ell-1,\ell].
\]
By (\ref{tens_path}), the path ${\mathcal{W}}$ coincides with the one defined
in \S \ \ref{subsec_RandPath}.

\begin{proposition}
\ \label{Prop_util}

\begin{enumerate}
\item For any $\beta,\eta\in P$, one gets
\[
{\mathbb{P}}(W_{\ell+1}=\beta\mid W_{\ell}=\eta)=K_{\kappa,\beta-\eta,}%
\frac{\tau^{\kappa+\eta-\beta}}{S_{\kappa}(\tau)}.
\]

\item Consider $\lambda,\mu\in P^{+}$ we have
\[
{\mathbb{P}}(W_{\ell}=\lambda,W_{0}=\mu,{\mathcal{W}}(t)\in{\mathcal{C}%
}\ {\text{ for any }}t\in\lbrack0,\ell])=f_{\lambda/\mu}^{\ell}\frac
{\tau^{\ell\kappa+\mu-\lambda}}{S_{\kappa}(\tau)^{\ell}}.
\]
In particular
\[
{\mathbb{P}}(W_{\ell+1}=\lambda,W_{\ell}=\mu,{\mathcal{W}}(t)\in{\mathcal{C}%
}\ {\text{ for any }}t\in\lbrack\ell,\ell+1])=m_{\mu,\kappa}^{\lambda}%
\frac{\tau^{\kappa+\mu-\lambda}}{S_{\kappa}(\tau)}.
\]

\end{enumerate}
\end{proposition}

\begin{proof}
1. We have
\[
{\mathbb{P}}(W_{\ell+1}=\beta\mid W_{\ell}=\eta)=\sum_{\pi\in B(b_{\pi
})_{\beta-\eta}}p_{\pi}
\]
where $B(b_{\pi})_{\beta-\eta}$ is the set of paths in $B(b_{\pi})$ of weight
$\beta-\eta$. We conclude noticing that all the paths in $B(b_{\pi}%
)_{\beta-\eta}$ have the same probability $\frac{\tau^{\kappa+\eta-\beta}%
}{S_{\kappa}(\tau)}$ and ${\mathrm{card}}(B(b_{\pi})_{\beta-\eta}%
)=K_{\kappa,\beta-\eta}$.

2. By Assertion 7 of Theorem \ref{Th_Littel}, we know that the number of paths
in $B(\pi_{\mu})\ast B^{\ast\ell}(\pi_{\kappa})$ starting at $\mu$, ending at
$\lambda$ and remaining in ${\mathcal{C}}$ is equal to $f_{\lambda/\mu}^{\ell
}$. Since the map $\Delta$ defined in (\ref{DefDelta}) is a bijection, the
integer $f_{\lambda/\mu}^{\ell}$ is also equal to the number of paths in
$B(\pi_{\mu})\otimes B^{\otimes\ell}(\pi_{\kappa})$ starting at $\mu$, ending
at $\lambda$ and remaining in ${\mathcal{C}}$. Moreover, each such path has
the form $b=b_{\mu}\otimes b_{1}\otimes\cdots\otimes b_{\ell}$ where
$b_{1}\otimes\cdots\otimes b_{\ell}\in B^{\otimes\ell}(\pi_{\kappa})$ has
weight $\lambda-\mu$. Therefore we have $p_{b}=\frac{\tau^{\ell\kappa
+\mu-\lambda}}{S_{\kappa}(\tau)^{\ell}}$.
\end{proof}

\subsection{The generalized Pitman transform}

\label{subsec_Pitman}By Assertion 8 of Theorem \ref{Th_Littel}, we know that
$B^{\otimes\ell}(\pi_{\kappa})$ is contained in ${\mathcal{P}}_{\min
{\mathbb{Z}}}$.\ Therefore, if we consider a path $b\in B^{\otimes\ell}%
(\pi_{\kappa}),$ its connected component $B(b)$ is contained in ${\mathcal{P}%
}_{\min{\mathbb{Z}}}$.\ Now, if $\eta\in B(b)$ is such that $\tilde{e}%
_{i}(\eta)=0$ for any $i=1,\ldots,n$, we should have $\operatorname{Im}%
\eta\subset{\mathcal{C}}$ by Assertion 3 of Proposition \ref{Prop_HP};
Assertion 1 of Theorem \ref{Th_Littel} thus implies that $\eta$ is the unique
path in $B(b)=B(\eta)$ such that $\tilde{e}_{i}(\eta)=0$ for any
$i=1,\ldots,n$.\ This permits to define the \emph{generalized Pitman}
transform on $B^{\otimes\ell}(\pi_{\kappa})$ as the map ${\mathfrak{P}}$ which
associates to any $b\in B^{\otimes\ell}(\pi_{\kappa})$ the unique path
${\mathfrak{P}}(b)\in B(b)$ such that $\tilde{e}_{i}(\eta)=0$ for any
$i=1,\ldots,n$.\ By definition, we have $\operatorname{Im}{\mathfrak{P}%
}(b)\subset{\mathcal{C}}$ and ${\mathfrak{P}}(b)(\ell)\in P_{+}$.

Let ${\mathcal{W}}$ be the random path of \S \ \ref{Subsec-RPW}.\ We define
the random process ${\mathcal{H}}$ setting
\begin{equation}
{\mathcal{H}}(t)={\mathfrak{P}}(\Pi_{\ell}({\mathcal{W}})){\mathcal{(}%
}t){\text{ for any }}t\in\lbrack\ell-1,\ell]. \label{defH}%
\end{equation}
For any $\ell\geq1$, we set $H_{\ell}:={\mathcal{H}}(\ell)$; one gets the

\begin{theorem}
\label{Th_LawH}The random sequence $H:=(H_{\ell})_{\ell\geq1}$ is a Markov
chain with transition matrix
\begin{equation}
\Pi(\mu,\lambda)=\frac{S_{\lambda}(\tau)}{S_{\kappa}(\tau)S_{\mu}(\tau)}%
\tau^{\kappa+\mu-\lambda}m_{\mu,\kappa}^{\lambda} \label{MatrixH}%
\end{equation}
where $\lambda,\mu\in P_{+}$.
\end{theorem}

\begin{proof}
Consider $\mu=\mu^{(\ell)},\mu^{(\ell-1)},\ldots,\mu^{(1)}$ a sequence of
elements in $P_{+}$.\ Let ${\mathcal{S}}(\mu^{(1)},\ldots\mu^{(\ell)}%
,\lambda)$ be the set of paths $b^{h}\in B^{\otimes\ell}(\pi_{\kappa})$
remaining in ${\mathcal{C}}$ and such that $b^{h}(k)=\mu^{(k)},k=1,\ldots
,\ell$ and $b^{(\ell+1)}=\lambda$. Consider $b=b_{1}\otimes\cdots\otimes
b_{\ell}\otimes b_{\ell+1}\in B^{\otimes\ell+1}(\pi_{\kappa})$. We have
${\mathfrak{P}}(b_{1}\otimes\cdots\otimes b_{k})(k)=\mu^{(k)}$ for any
$k=1,\ldots,\ell$ and ${\mathfrak{P}}(b)(\ell+1)=\lambda$ if and only if
${\mathfrak{P}}(b)\in{\mathcal{S}}(\mu^{(1)},\ldots\mu^{(\ell)},\lambda)$.
Moreover, by (\ref{potimesell}), for any $b^{h}\in S(\mu^{(1)},\ldots
\mu^{(\ell)},\lambda)$, we have $\displaystyle {\mathbb{P}}(b\in
B(b^{h}))=\sum_{b\in B(b^{h})}p_{b}=\sum_{b\in B(b^{h})}\frac{\tau
^{(\ell+1)\kappa-{\mathrm{wt}}(b)}}{S_{\kappa}(\tau)^{\ell+1}};$ combining
(\ref{slambda}) and (\ref{defS}), one obtains $\displaystyle {\mathbb{P}}(b\in
B(b^{h}))=\frac{\tau^{(\ell+1)\kappa-\lambda}S_{\lambda}(\tau)}{S_{\kappa
}(\tau)^{\ell+1}},$ which only depends on $\lambda$. This gives\
\begin{align*}
{\mathbb{P}}(H_{\ell+1}=\lambda,H_{k}=\mu^{(k)},\forall k=1,\ldots,\ell)  &
=\sum_{b^{h}\in{\mathcal{S}}(\mu^{(1)},\ldots\mu^{(\ell)},\lambda)}\sum_{b\in
B(b^{h})}p_{b}\\
&  ={\mathrm{card}}({\mathcal{S}}(\mu^{(1)},\ldots\mu^{(\ell)},\lambda
))\frac{\tau^{(\ell+1)\kappa-\lambda}S_{\lambda}(\tau)}{S_{\kappa}(\tau
)^{\ell+1}}.
\end{align*}
By assertion 9 of Theorem \ref{Th_Littel} and an easy induction, we have also
\[
{\mathrm{card}}({\mathcal{S}}(\mu^{(1)},\ldots\mu^{(\ell)},\lambda
))=\prod_{k=1}^{\ell-1}m_{\mu^{(k)},\kappa}^{\mu^{(k+1)}}\times m_{\mu,\kappa
}^{\lambda}.
\]
We thus get
\[
{\mathbb{P}}(H_{\ell+1}=\lambda,H_{k}=\mu^{(k)},\forall k=1,\ldots,\ell
)=\prod_{k=1}^{\ell-1}m_{\mu^{(k)},\kappa}^{\mu^{(k+1)}}\times m_{\mu,\kappa
}^{\lambda}\frac{\tau^{(\ell+1)\kappa-\lambda}S_{\lambda}(\tau)}{S_{\kappa
}(\tau)^{\ell+1}}.
\]
Similarly
\[
{\mathbb{P}}(H_{k}=\mu^{(k)},\forall k=1,\ldots,\ell)=\prod_{k=1}^{\ell
-1}m_{\mu^{(k)},\kappa}^{\mu^{(k+1)}}\frac{\tau^{\ell\kappa-\mu}S_{\lambda
}(\tau)}{S_{\kappa}(\tau)^{\ell}},
\]
this readily implies
\begin{align*}
{\mathbb{P}}(H_{\ell+1}=\lambda\mid H_{k}=\mu^{(k)},\forall k=1,\ldots,\ell)
&  =\frac{{\mathbb{P}}(H_{\ell+1}=\lambda,H_{k}=\mu^{(k)},\forall
k=1,\ldots,\ell)}{{\mathbb{P}}(H_{k}=\mu^{(k)},\forall k=1,\ldots,\ell)}\\
&  =\frac{S_{\lambda}(\tau)}{S_{\kappa}(\tau)S_{\mu}(\tau)}\tau^{\kappa
+\mu-\lambda}m_{\mu,\kappa}^{\lambda}.
\end{align*}

\end{proof}

\section{Symmetrization}

In \S \ \ref{Probability distribution on elementary paths}, we have chosen a
probability distribution $p$ on a crystal $B(\pi_{\kappa})$ where $\kappa\in
P_{+}$ and $\pi_{\kappa}$ is an elementary path from $0$ to $\kappa$ remaining
in the cone ${\mathcal{C}}$.\ This distribution depends on $\tau\in
{\mathbb{R}}_{>0}^{n}$ and Proposition \ref{Prop_Conv} gives a sufficient
condition to ensure that $S_{\kappa}(\tau)$ is finite.\ Since the characters
of the highest weight representations are symmetric under the action of the
Weyl group, it is possible to define, starting from the distribution $p$ and
for each $w$ in the Weyl group ${\mathsf{W}}$ of ${\mathfrak{g}}$, a
probability distribution $p_{w}$ which reflects this symmetry.

\subsection{Twisted probability distribution}

Recall $\tau=(\tau_{1},\ldots,\tau_{n})\in{\mathcal{T}}$ is fixed. Given any
$w\in{\mathsf{W}}$, we want to define a probability distribution on
$B(\kappa)$ for each $w\in{\mathsf{W}}$. Recall that $w(\alpha_{i})$ is a
(real) root of ${\mathfrak{g}}$ for any $w\in{\mathsf{W}}$ and any simple root
$\alpha_{i}$; this root is neither simple or even positive in general. By
general properties of the root systems, we know that $w(\alpha_{i})$ can be
decomposed as follows
\[
w(\alpha_{i})=\left\{
\begin{array}
[c]{c}%
\alpha_{k_{1}}+\cdots+\alpha_{k_{r}}\\
{\text{or}}\\
-(\alpha_{k_{1}}+\cdots+\alpha_{k_{r}})
\end{array}
\right.
\]
where $\alpha_{k_{1}},\ldots,\alpha_{k_{r}}$ are simple roots depending on
$w$. Let us define the $n$-tuple $\tau^{w}=(\tau_{1}^{w},\ldots,\tau_{n}%
^{w})\in{\mathbb{R}}_{>0}^{n}$ setting
\[
\tau_{i}^{w}=\left\{
\begin{array}
[c]{lll}%
\prod_{s=1}^{r}\tau_{k_{s}} & {\text{ if }} & w(\alpha_{i})=\alpha_{k_{1}%
}+\cdots+\alpha_{k_{r}},\\
\prod_{s=1}^{r}\tau_{k_{s}}^{-1} & {\text{ if }} & w(\alpha_{i})=-(\alpha
_{k_{1}}+\cdots+\alpha_{k_{r}}),
\end{array}
\right.
\]
that is
\begin{equation}
\tau_{i}^{w}=\tau^{w(\alpha_{i})}. \label{ti(w)}%
\end{equation}
More generally for any $\bar{u}=u_{1}\alpha_{1}+\cdots+u_{n}\alpha_{n}\in Q,$
we have
\[
(\tau^{w})^{\bar{u}}=(\tau_{1}^{w})^{u_{1}}\cdots(\tau_{n}^{w})^{u_{n}}%
=\tau^{w(\bar{u})}.
\]
Observe also that $\tau^{w}\notin{\mathcal{T}}$ in general; indeed we have the following

\begin{lemma}
\label{Lem_ti}$\tau(w)\in{\mathcal{T}}$ if and only if $w=1$.
\end{lemma}

\begin{proof}
It suffices to show that for any $w\in{\mathsf{W}}\setminus\{Id\}$ distinct
from the identity, there is at least a simple root $\alpha_{i}$ such that
$w(\alpha_{i})=-(\alpha_{k_{1}}+\cdots+\alpha_{k_{r}})\in-Q_{+}$. Indeed, we
will have in that case $\tau_{i}^{w}=\frac{1}{\tau_{k_{1}}\cdots\tau_{k_{r}}%
}>1$ since $\tau(w)\in{\mathcal{T}}$. Consider $w\in{\mathsf{W}}%
\setminus\{Id\}$ such that $w(\alpha_{i})\in Q_{+}$ for any $i=1,\ldots,\ell$.
Let us decompose $w$ as a reduced word $w=s_{i_{1}}\cdots s_{i_{t}}$; by lemma
3.11 in \cite{KacB}, we must have $w(\alpha_{i_{t}})\in-Q_{+}$, hence a
contradiction. This means that $w=1$.
\end{proof}

\bigskip

Consider $\kappa\in P_{+}$.\ Recall that we have by definition $s_{\kappa
}=e^{\kappa}S_{\kappa}(T_{1},\ldots,T_{n})$ where $T_{i}=e^{-\alpha_{i}}$.
Since $s_{\kappa}$ is symmetric under ${\mathsf{W}}$, we have $s_{\kappa
}=e^{w(\kappa)}S_{\kappa}(T_{1}^{w},\ldots,T_{n}^{w})$ with $T_{i}%
^{w}=e^{-w(\alpha_{i})}$ for any $i=1,\ldots,n$.\ Therefore
\[
S_{\kappa}(T_{1}^{w},\ldots,T_{n}^{w})=e^{\kappa-w(\kappa)}S_{\kappa}%
(T_{1},\ldots,T_{n})\ {\text{ for any }}w\in{\mathsf{W}}.
\]
Since $\kappa-w(\kappa)$ belongs to $Q^{+}$, we can specialize each $T_{i}$ in
$\tau_{i}$. Then $T_{i}^{w}$ is specialized in $\tau_{i}^{w}$ and we get
\begin{equation}
S_{\kappa}(\tau^{w})=\tau^{w(\kappa)-\kappa}S_{\kappa}(\tau),
\label{rela_twist}%
\end{equation}
in particular, it is finite.

\begin{definition}
For any $w\in{\mathsf{W}}$ and any integer $\ell\geq1$, let $p^{w}$ be the
probability distribution on $B(\pi_{\kappa})^{\otimes\ell}$ defined by: for
any $b\in B(\pi_{\kappa})^{\otimes\ell}$
\[
p_{b}^{w}:=\frac{(\tau^{w})^{\ell\kappa-{\mathrm{wt}}(b)}}{S_{\kappa}(\tau
^{w})^{\ell}}=\frac{\tau^{\ell w(\kappa)-{\mathrm{wt}}(w(b))}}{S_{\kappa}%
(\tau^{w})^{\ell}}
\]
where $w(b)$ is the image of $b$ under the action of ${\mathsf{W}}$ (see
Assertion 6 of Theorem \ref{Th_Littel}). In particular, $p^{1}=p$ coincides
with the probability distribution (\ref{potimesell}).
\end{definition}

The following lemma states that the probabilities $p^{w}$ and $p$ coincide up
to the permutation of the elements in $B(\pi_{\kappa})^{\otimes\ell}$ given by
the action of $w$ described in Assertion 6 of Theorem \ref{Th_Littel}.

\begin{lemma}
\label{Lem_fund}For any $w\in{\mathsf{W}}$ and any $b\in B(\pi_{\kappa
})^{\otimes\ell}$, we have $\ p_{b}^{w}=p_{w(b)},\ $ where $w(b)$ is the image
of $b$ under the action of ${\mathsf{W}}$ (see Assertion 6 of Theorem
\ref{Th_Littel}).
\end{lemma}

\begin{proof}
Recall that ${\mathrm{wt}}(w(b))=w({\mathrm{wt}}(b))$; therefore
$\displaystyle p_{w(b)}=\frac{\tau^{\ell\kappa-{\mathrm{wt}}(w(b))}}%
{S_{\kappa}(\tau)^{\ell}}.$ On the other hand, by (\ref{rela_twist}) we have
$\displaystyle p_{b}^{w}:=\frac{\tau^{\ell w(\kappa)-{\mathrm{wt}}(w(b))}%
}{S_{\kappa}(\tau^{w})^{\ell}}=\frac{\tau^{\ell w(\kappa)-{\mathrm{wt}}%
(w(b))}}{\tau^{\ell w(\kappa)-\ell\kappa}S_{\kappa}(\tau)^{\ell}}$ and the
equality $p_{b}^{w}=p_{w(b)}$ follows.
\end{proof}

\subsection{Twisted random paths}

Let $w\in{\mathsf{W}}$ and denote by $X^{w}$ the random variable defined on
$(B(\pi_{\kappa}),p^{w})$ with law given by:
\[
{\mathbb{P}}(X^{w}=\pi)=p_{\pi}^{w}=p_{w(\pi)}\quad{\text{for all}}\quad\pi\in
B(\pi_{\kappa}).
\]

Set $m^{w}:={\mathbb{E}}(X^{w})$ and $m:=m^{1}$.

\begin{proposition}
Assume $\tau\in{\mathcal{T}}.$ One gets

\begin{enumerate}
\item $m(1)\in\mathring{{\mathcal{C}}}$,

\item $m^{w}=w^{-1}(m)$,

\item $m^{w}(1)\in\mathring{{\mathcal{C}}}$ if and only if $w$ is equal to the identity.
\end{enumerate}
\end{proposition}

\begin{proof}
1. By definition of $\mathring{\mathcal{C}}$, we have to prove that
$h_{i}(m(1))>0$ for any $i=1,\ldots,n$.\ Recall that $m=\sum_{\pi\in
B(\pi_{\kappa})}p_{\pi}\pi$; observe that the quantity
\[
c_{i}=h_{i}(m(1))=\sum_{\pi\in B(\pi_{\kappa})}p_{\pi}h_{i}(\pi(1))
\]
is well defined by Corollary \ref{Cor_mexists}. We can decompose the crystal
$B(\pi_{\kappa})$ in its $i$-chains, that is the sub-crystal obtained by
deleting all the arrows $j\neq i$. When ${\mathfrak{g}}$ is not of finite
type, the lengths of these $i$-chains are all finite but not bounded.\ The
contribution to $c_{i}$ of any $i$-chain $C:a_{0}\overset{i}{\rightarrow}%
a_{1}\overset{i}{\rightarrow}\cdots\overset{i}{\rightarrow}a_{k}$ of length
$k$ is equal to $\displaystyle c_{i}(C)=\sum_{j=0}^{k}p_{a_{i}}h_{i}%
({\mathrm{wt}}(a_{j})).$ Since $\tilde{e}_{i}(a_{0})=0$ and $\tilde{f}%
_{i}^{k+1}(a_{0})=0$, we obtain $h_{i}({\mathrm{wt}}(a_{0}))=k$. By definition
of the distribution $p$ and Proposition \ref{prop-ac_etil}, we have the
relation $p_{a_{j}}=\tau_{i}^{j}p_{a_{0}}$.\ Finally, we get
\[
c_{i}(C)=p_{a_{0}}\sum_{j=0}^{k}\tau_{i}^{j}(k-2j)=p_{a_{0}}\sum
_{j=0}^{\left\lfloor k/2\right\rfloor }(k-2j)(\tau_{i}^{j}-\tau_{i}^{k-j}).
\]

In particular the hypothesis $\tau_{i}\in]0,1[$ for any $i=1,\ldots,n$ implies
that $c_{i}(C)>0$ for any $i$-chain of length $k>0$; one thus gets $c_{i}>0$
noticing that $B(\pi_{k})$ contains at least an $i$-chain of length $k>0$,
otherwise the action of the Chevalley generators $e_{i},f_{i}$ on the
irreducible module $V(\pi_{\lambda})$ would be trivial.

2. By Lemma \ref{Lem_fund}, we can write
\[
m^{w}=\sum_{\pi\in B(\pi_{\kappa})}p_{w(\pi)}\pi=\sum_{\pi^{\prime}\in
B(\pi_{\kappa})}p_{\pi^{\prime}}w^{-1}(\pi^{\prime})=w^{-1}\left(  \sum
_{\pi^{\prime}\in B(\pi_{\kappa})}p_{\pi^{\prime}}\pi^{\prime}\right)
=w^{-1}(m)
\]
where we use assertion 7 of Theorem \ref{Th_Littel} in the third equality.

3. Since $m^{w}=w^{-1}(m)$ and $m(1)\in\mathring{{\mathcal{C}}}$, one gets
$m^{w}(1)\notin{\mathcal{C}}$ because ${\mathcal{C}}$ is a fundamental domain
for the action of the Weyl group ${\mathsf{W}}$ on the tits cone
${\mathcal{X}}=\cup_{w\in{\mathsf{W}}}w({\mathcal{C}})$ (see \cite{KacB}
Proposition 3.12).
\end{proof}

\bigskip

Now let $X^{w}=(X_{\ell}^{w})_{\ell\geq1}$ be a sequence of i.i.d. random
variables defined on $B(\pi_{\kappa})$ with probability distribution $p^{w}%
$.\ The random process ${\mathcal{W}}^{w}=({\mathcal{W}}_{t}^{w})_{t>0}$ is
defined by: for all $\ell\geq1$ and $t\in\lbrack\ell-1,\ell]$
\[
{\mathcal{W}}^{w}(t):=\Pi_{\ell}(X^{w})(t)=X_{1}^{w}\otimes X_{2}^{w}%
\otimes\cdots\otimes X_{\ell-1}^{w}\otimes X_{\ell}^{w}(t).
\]
By (\ref{tens_path}), the random walk $W^{w}$ is defined as in
\S \ \ref{subsec_RandPath} from ${\mathcal{W}}^{w}$. For any $\ell
\in{\mathbb{Z}}_{\geq0}$, we also define the function $\psi_{\ell}^{w}$ on
$P_{+}$ setting
\[
\psi_{\ell}^{w}(\mu):={\mathbb{P}}_{\mu}({\mathcal{W}}^{w}(t)\in{\mathcal{C}%
}\ {\text{ for any }}t\in\lbrack0,\ell]).
\]
The quantity $\psi_{\ell}^{w}(\mu)$ is equal to the probability of the event
\textquotedblleft${\mathcal{W}}^{w}$ starting at $\mu$ remains in the cone
${\mathcal{C}}$ until the instant $\ell$\textquotedblright. We also introduce
the function
\[
\psi^{w}(\mu):={\mathbb{P}}_{\mu}({\mathcal{W}}^{w}(t)\in{\mathcal{C}}%
,t\geq0).
\]
For $w=1$, we simply write $\psi$ and $\psi_{\ell}$ instead of $\psi^{1}$ and
$\psi_{\ell}^{1}$.

The following proposition is a consequence of the previous lemma, Proposition
\ref{Prop_lim} and Corollary \ref{Cor_mexists}.

\begin{proposition}
\label{Prop_lim2} \ 

\begin{enumerate}
\item We have $\lim_{\ell\rightarrow+\infty}\psi_{\ell}^{w}(\mu)=\psi^{w}%
(\mu)$ for any $\mu\in P_{+}.$

\item If $w\neq1$, then $\psi^{w}(\mu)=0$ for any $\mu\in P_{+}.$

\item If $w=1$, then $\psi(\mu)\geq0$ for any $\mu\in P_{+}$.
\end{enumerate}
\end{proposition}

Similarly to Proposition \ref{Prop_util} and using (\ref{rela_twist}), we
obtain the

\begin{proposition}
\ 

\begin{enumerate}
\item For any weights $\beta$ and $\eta$, one gets
\[
{\mathbb{P}}(W_{\ell+1}^{w}=\beta\mid W_{\ell}^{w}=\eta)=K_{\kappa,\beta
-\eta,}\frac{\tau^{w(\kappa+\eta-\beta)}}{S_{\kappa}(\tau^{w})}=K_{\kappa
,\beta-\eta,}\frac{\tau^{\kappa+w(\eta)-w(\beta)}}{S_{\kappa}(\tau)}.
\]

\item For any dominant weights $\lambda$ and $\mu$, one gets
\[
{\mathbb{P}}(W_{\ell}^{w}=\lambda,W_{0}^{w}=\mu,{\mathcal{W}}^{w}%
(t)\in{\mathcal{C}}\ {\text{ for any }}t\in\lbrack0,\ell])=f_{\lambda/\mu
}^{\ell}\frac{\tau^{w(\ell\kappa+\mu-\lambda)}}{S_{\kappa}(\tau^{w})^{\ell}%
}=f_{\lambda/\mu}^{\ell}\frac{\tau^{\ell\kappa+w(\mu)-w(\lambda)}}{S_{\kappa
}(\tau)^{\ell}}.
\]
In particular
\[
{\mathbb{P}}(W_{\ell+1}^{w}=\lambda,W_{\ell}^{w}=\mu,{\mathcal{W}}^{w}%
(t)\in{\mathcal{C}}\ {\text{ for any }}t\in\lbrack\ell,\ell+1])=m_{\mu,\kappa
}^{\lambda}\frac{\tau^{w(\kappa+\mu-\lambda)}}{S_{\kappa}(\tau^{w})}%
=m_{\mu,\kappa}^{\lambda}\frac{\tau^{\kappa+w(\mu)-w(\lambda)}}{S_{\kappa
}(\tau)}.
\]

\end{enumerate}
\end{proposition}

\section{Law of the conditioned random path}

\subsection{The harmonic function $\psi$}

\label{Subsec_Harm}

By Assertion 2 of the previous proposition, we can write%

\begin{equation}
\psi_{\ell}^{w}(\mu)={\mathbb{P}}_{\mu}({\mathcal{W}}^{w}(t)\in{\mathcal{C}%
}\ {\text{ for any }}t\in\lbrack0,\ell])=\sum_{\lambda\in P_{+}}f_{\lambda
/\mu}^{\ell}\frac{\tau^{\ell\kappa+w(\mu)-w(\lambda)}}{S_{\kappa}(\tau)^{\ell
}} \label{psi(l,w)}%
\end{equation}
where $f_{\lambda/\mu}^{\ell}$ is the number of highest weight vertices in the
crystal
\begin{equation}
B(\mu)\otimes B(\kappa)^{\otimes\ell}\simeq{\displaystyle\bigoplus
\limits_{\lambda\in P_{+}}}B(\lambda)^{\oplus f_{\lambda/\mu}^{\ell}}.
\label{dec}%
\end{equation}
By interpreting (\ref{dec}) in terms of characters, we get
\begin{equation}
s_{\mu}\times s_{\kappa}^{\ell}=\sum_{\lambda\in P_{+}}f_{\lambda/\mu}^{\ell
}s_{\lambda}. \label{prod}%
\end{equation}
The Weyl character formula
\[
s_{\lambda}=\frac{\sum_{w\in{\mathsf{W}}}\varepsilon(w)e^{w(\lambda+\rho
)-\rho}}{\prod_{\alpha\in R_{+}}(1-e^{-\alpha})^{m_{\alpha}}}%
\]
yields
\[
\prod_{\alpha\in R_{+}}(1-e^{-\alpha})^{m_{\alpha}}s_{\mu}\times s_{\kappa
}^{\ell}=\sum_{\lambda\in P_{+}}f_{\lambda/\mu}^{\ell}\sum_{w\in{\mathsf{W}}%
}\varepsilon(w)e^{w(\lambda+\rho)-\rho}.
\]
In the previous formal series in ${\mathbb{C}}[[P]]$, all the monomials
$e^{w(\lambda+\rho)-\rho}$ with $w\in W$ and $\lambda\in P_{+}$ are distinct
(see \cite{KacB} Proposition 3.12). We thus also have
\begin{equation}
\prod_{\alpha\in R_{+}}(1-e^{-\alpha})^{m_{\alpha}}s_{\mu}\times s_{\kappa
}^{\ell}=\sum_{w\in{\mathsf{W}}}\varepsilon(w)\sum_{\lambda\in P_{+}%
}f_{\lambda/\mu}^{\ell}e^{w(\lambda+\rho)-\rho} \label{ThetaPsi0}%
\end{equation}
or equivalently%
\[
\prod_{\alpha\in R_{+}}(1-e^{-\alpha})^{m_{\alpha}}S_{\mu}\times S_{\kappa
}^{\ell}=\sum_{w\in{\mathsf{W}}}\varepsilon(w)\sum_{\lambda\in P_{+}%
}f_{\lambda/\mu}^{\ell}e^{w(\lambda+\rho)-\rho-\ell\kappa-\mu}.
\]
We now need the following lemma.

\begin{lemma}
For any $w\in{\mathsf{W}}$ and $\mu\in P_{+}$, set $\displaystyle \Pi_{\ell
}^{w}(\mu):=\sum_{\lambda\in P_{+}}f_{\lambda/\mu}^{\ell}\frac{\tau
^{\ell\kappa+\rho+\mu-w(\lambda+\rho)}}{S_{\kappa}(\tau)^{\ell}}.$

We then have $\displaystyle \lim_{\ell\rightarrow+\infty}\Pi_{\ell}^{w}%
(\mu)=0$ when $w\neq1$ and the series $\sum_{w\in{\mathsf{W}}}\varepsilon
(w)\Pi_{\ell}^{w}(\mu)$ converges uniformly in $\ell$.
\end{lemma}

\begin{proof}
Using (\ref{psi(l,w)}), one gets
\begin{equation}
\Pi_{\ell}^{w}(\mu)=\tau^{\rho-w(\rho)+\mu-w(\mu)}\sum_{\lambda\in P_{+}%
}f_{\lambda/\mu}^{\ell}\frac{\tau^{\ell\kappa+w(\mu)-w(\lambda)}}{S_{\kappa
}(\tau)^{\ell}}=\tau^{\rho-w(\rho)+\mu-w(\mu)}\psi_{\ell}^{w}(\mu).
\label{ThetaPsi}%
\end{equation}

Fix $w\neq1$. Since $\tau^{\rho-w(\rho)+\mu-w(\mu)}$ does not depend on $\ell$
and $\lim_{\ell\rightarrow+\infty}\psi_{\ell}^{w}(\mu)=0$ by Proposition
\ref{Prop_lim2}, we derive $\lim_{\ell\rightarrow+\infty}\Pi_{\ell}^{w}%
(\mu)=0$ as desired.

Now, we have obviously $0\leq\psi_{\ell}^{w}(\mu)\leq1$ and the series
$\sum_{w\in{\mathsf{W}}}\tau^{\rho-w(\rho)+\mu-w(\mu)}$ converges by Corollary
\ref{Cor_Cvm}.\ The uniform convergence in $\ell$ of the series $\sum
_{w\in{\mathsf{W}}}\varepsilon(w)\Pi_{\ell}^{w}(\mu)$ thus follows from the
inequality $\left\vert \varepsilon(w)\Pi_{\ell}^{w}(\mu)\right\vert \leq
\tau^{\rho-w(\rho)+\mu-w(\mu)}$, which is a direct consequence of
(\ref{ThetaPsi}).
\end{proof}

\bigskip

We can now set $\tau_{i}=e^{-\alpha_{i}}$ in (\ref{ThetaPsi0}) and get
\begin{equation}
\prod_{\alpha\in R_{+}}(1-\tau^{\alpha})^{m_{\alpha}}S_{\mu}(\tau)=\sum
_{w\in{\mathsf{W}}}\varepsilon(w)\sum_{\lambda\in P_{+}}f_{\lambda/\mu}^{\ell
}\frac{\tau^{\ell\kappa+\rho+\mu-w(\lambda+\rho)}}{S_{\kappa}(\tau)^{\ell}%
}{\text{.}} \label{identity}%
\end{equation}
Consequently, we have
\[
\prod_{\alpha\in R_{+}}(1-\tau^{\alpha})^{m_{\alpha}}S_{\mu}(\tau)=\sum
_{w\in{\mathsf{W}}}\varepsilon(w)\Pi_{\ell}^{w}(\mu)=\Pi_{\ell}^{1}(\mu
)+\sum_{w\neq1}\varepsilon(w)\Pi_{\ell}^{w}(\mu)
\]
with $\Pi_{\ell}^{1}(\mu)=\psi_{\ell}(\mu),$ by (\ref{ThetaPsi}). Letting
$\ell\rightarrow+\infty$, the previous lemma finally gives
\[
\psi(\mu)=\prod_{\alpha\in R_{+}}(1-\tau^{\alpha})^{m_{\alpha}}S_{\mu}(\tau).
\]

We have established the following theorem, which is the analogue in our
context of Corollary 7.4.3 in \cite{LLP}:

\begin{theorem}
\label{Th_PSi}For any $\mu\in P_{+}$, we have
\[
\psi(\mu)={\mathbb{P}}_{\mu}({\mathcal{W}}(t)\in{\mathcal{C}}{\text{ for any
}}t\geq0)=\prod_{\alpha\in R_{+}}(1-\tau^{\alpha})^{m_{\alpha}}S_{\mu}%
(\tau)>0
\]
In particular, the harmonic function $\psi$ is positive and does not depend on
the dominant weight $\kappa$ considered.
\end{theorem}

\begin{corollary}
The law of the random walk $W$ conditioned by the event
\[
E:=({\mathcal{W}}(t)\in{\mathcal{C}}\ {\text{\textbf{for any }}}t\geq0)
\]
is the same as the law of the Markov chain $H$ defined as the generalized
Pitman transform of $W$ (see Theorem \ref{Th_LawH}). In particular, this law
only depends on $\kappa$ and not on the choice of the path $\pi_{\kappa}$ such
that $\operatorname{Im}\pi_{\kappa}\subset{\mathcal{C}}$.
\end{corollary}

\begin{proof}
Let $\Pi$ be the transition matrix of $W$ and $\Pi^{E}$ its restriction to the
event $E$.\ We have seen in \S \ \ref{subsec-Markov} that the transition
matrix of $W$ conditioned by $E$ is the $h$-transform of $\Pi^{E}$ by the
harmonic function
\[
h_{E}(\mu):={\mathbb{P}}_{\mu}({\mathcal{W}}(t)\in{\mathcal{C}}{\text{ for any
}}t\geq0).
\]
By the previous theorem, we have $h_{E}=\psi$. It also easily follows from
Theorem \ref{Th_LawH} that the transition matrix of $H$ is the $\psi
$-transform of $\Pi^{E}$. Therefore both $H$ and the conditioning of $W$ by
$E$ have the same law.
\end{proof}

\subsection{Random walks defined from non irreducible representations}

For simplicity we restrict ourselves in this paragraph to the case where
${\mathfrak{g}}$ is a (finite-dimensional) Lie algebra with (invertible)
Cartan matrix $A$. In particular, $m_{\alpha}=1$ for any $\alpha\in R_{+}$.
Consider $\tau=(\tau_{1},\ldots,\tau_{n})\in{\mathcal{T}}$. Then both root and
weight lattices have the same rank $n$.\ Moreover, the Cartan matrix $A$ is
the transition matrix between the weight and root lattices. In particular,
each weight $\beta\in P$ decomposes on the basis of simple roots as
$\beta=\beta_{1}^{\prime}\alpha_{1}+\cdots\beta_{n}^{\prime}\alpha_{n}$ where
$(\beta_{1}^{\prime},\ldots,\beta_{n}^{\prime})\in\frac{1}{\det A}{\mathbb{Z}%
}^{n}$ and we can set $\tau^{\beta}=\tau_{1}^{\beta_{1}^{\prime}}\cdots
\tau_{n}^{\beta_{n}^{\prime}}$.

Let $M$ be a finite dimensional ${\mathfrak{g}}$-module with decomposition in
irreducible components
\[
M\simeq{\displaystyle \bigoplus\limits_{\kappa\in\varkappa}}V(\kappa)^{\oplus
a_{\kappa}}
\]
where $\varkappa$ is a finite subset of$\ P_{+}$ and $a_{\kappa}>0$ for any
$\kappa\in\varkappa$. For each $\kappa\in\varkappa$ choose a path
$\eta_{\kappa}$ in $P$ from $0$ to $\kappa$ contained in ${\mathcal{C}}$.\ Let
$B(\varkappa)$ be the set of paths obtained by applying the operators
$\tilde{e}_{i}$,$\tilde{f}_{i},i=1,\ldots,n$ to the paths $\eta_{\kappa
},\kappa\in\varkappa$. This set is a realization of the crystal of the
${\mathfrak{g}}$-module $\oplus_{\kappa\in\varkappa}V(\kappa)$ (without
multiplicities) and we have
\[
B(\varkappa)=\bigsqcup\limits_{\kappa\in\varkappa}B(\eta_{\kappa}).
\]
Given $\pi=\pi_{1}\otimes\cdots\otimes\pi_{\ell}$ in $B^{\otimes\ell
}(\varkappa)$ such that $\pi_{a}\in B(\kappa_{a})$ for any $a=1,\ldots,\ell$,
we set $a_{\pi}=a_{\kappa_{1}}\cdots a_{\kappa_{\ell}}$. By formulas
(\ref{Ten_Prod}), the function $a$ is constant on the connected components of
$B^{\otimes\ell}(\varkappa)$.

We are going to define a probability distribution on $B(\varkappa)$ compatible
with its weight graduation and taking into account the multiplicities
$a_{\kappa}$.\ We cannot proceed as in (\ref{potimesell}) by working only with
the root lattice of ${\mathfrak{g}}$ since $B(\varkappa)$ contains fewer
highest weight paths.\ So the underlying lattice to consider is the weight
lattice. We first set
\[
\Sigma_{M}(\tau)=\sum_{\kappa\in\varkappa}\sum_{\pi\in B(\eta_{\kappa}%
)}a_{\kappa}\tau^{-{\mathrm{wt}}(\pi)}=\sum_{\pi\in B(\varkappa)}a_{\pi}%
\tau^{-{\mathrm{wt}}(\pi)}=\sum_{\kappa\in\varkappa}a_{\kappa}s_{\kappa}%
(\tau)=\sum_{\kappa\in\varkappa}a_{\kappa}\tau^{-\kappa}S_{\kappa}(\tau).
\]
We define the probability distribution $p$ on $B(\varkappa)$ by setting
$p_{\pi}=a_{\kappa}\frac{\tau^{-{\mathrm{wt}}(\pi)}}{\Sigma_{M}(\tau)}$ for
any $\pi\in B(\eta_{\kappa})$. When ${\mathrm{card}}(\varkappa)=1$, we recover
the probability distribution of
\S \ \ref{Probability distribution on elementary paths}.\ Observe that we
have
\[
\Sigma_{M}(\tau)^{\ell}=\sum_{\pi\in B^{\otimes\ell}(\varkappa)}a_{\pi}%
\tau^{-{\mathrm{wt}}(\pi)}{\text{ for any }}\ell\geq0{\text{.}}
\]
So we can define a probability distribution $p^{\otimes\ell}$ on
$B^{\otimes\ell}(\varkappa)$ such that
\[
p_{\pi}=a_{\pi}\frac{\tau^{-{\mathrm{wt}}(\pi)}}{\Sigma_{M}(\tau)^{\ell}%
}{\text{ for any }}\pi=\pi_{1}\otimes\cdots\otimes\pi_{\ell}\in B^{\otimes
\ell}(\varkappa).
\]
Let $X=(X_{\ell})_{\ell\geq1}$ be a sequence of i.i.d. random variables
defined on $B(\varkappa)$ with probability distribution $p$.\ The random
process ${\mathcal{W}}$ and the random walk $W$ are then defined from $X$ and
$p^{\otimes{\mathbb{N}}}$ as in \S \ \ref{subsec_RandPath}.

It is then possible to extend our results to the random path ${\mathcal{W}}$
and its corresponding random walk $W$ obtained from the set of elementary
paths $B(\varkappa)$. We have then%

\[
{\mathbb{P}}(W_{\ell+1}=\beta\mid W_{\ell}=\gamma)=\frac{K_{M,\beta-\gamma}%
}{\Sigma_{M}(\tau)}\tau^{\gamma-\beta}
\]
for any weights $\beta$ and $\gamma$ where $K_{M,\beta-\gamma}$ is the
dimension of the space of weight $\beta-\gamma$ in $M$. We indeed have
$K_{M,\beta-\gamma}=\sum_{\kappa\in\varkappa}a_{\kappa}K_{\kappa,\beta-\gamma
}$where $K_{\kappa,\beta-\gamma}$ is the number of paths $\eta\in B(\kappa)$
such that $\eta(1)=\beta-\gamma$.\ Given $\lambda$ and $\mu$ two dominant
weights, we also get
\begin{equation}
{\mathbb{P}}(W_{\ell+1}=\lambda\mid W_{\ell}=\mu,{\mathcal{W}}(t)\in
{\mathcal{C}}\ {\text{ for any }}t\in\lbrack\ell,\ell+1])=\frac{m_{M,\mu
}^{\lambda}}{\Sigma_{M}(\tau)}\tau^{\mu-\lambda} \label{SubstochM}%
\end{equation}
where $m_{M,\mu}^{\lambda}$ is the multiplicity of $V(\lambda)$ in $M\otimes
V(\mu)$. We indeed have $m_{M,\mu}^{\lambda}=\sum_{\kappa\in\varkappa
}a_{\kappa}m_{\kappa,\mu}^{\lambda}$where $m_{\kappa,\mu}^{\lambda}$ is the
number of paths $\eta\in B(\kappa)$ such that $\eta(1)=\lambda-\mu$ which
remains in ${\mathcal{C}}$.\ 

We define the generalized Pitman transform ${\mathfrak{P}}$ and the Markov
chain $H$ as in \S \ \ref{subsec_Pitman}. For any $\ell\geq1$, we yet write
$\psi_{\ell}(\mu)={\mathbb{P}}_{\mu}({\mathcal{W}}(t)\in{\mathcal{C}}$ for any
$t\in\lbrack1,\ell])$. We then have
\begin{multline*}
\psi_{\ell}(\mu)=\sum_{\pi\in B^{\otimes\ell}(\varkappa),\mu+\pi
(t)\in{\mathcal{C}}{\text{ for }}t\in\lbrack0,\ell]}p_{\pi}=\sum_{\lambda\in
P_{+}}\ \sum_{\pi\in B^{\otimes\ell}(\varkappa),{\text{ }}\mu+\pi
(t)\in{\mathcal{C}}{\text{ for }}t\in\lbrack0,\ell],\pi(\ell)=\lambda}a_{\pi
}\frac{\tau^{\mu-\lambda}}{\Sigma_{M}(\tau)^{\ell}}=\\
\sum_{\lambda\in P_{+}}f_{\lambda/\mu}^{\ell}\frac{\tau^{\mu-\lambda}}%
{\Sigma_{M}(\tau)^{\ell}}%
\end{multline*}
where $f_{\lambda/\mu}^{\ell}$ is the multiplicity of $V(\lambda)$ in
$V(\mu)\otimes M^{\otimes\ell}$. We indeed have the equality $f_{\lambda/\mu
}^{\ell}=\sum_{\pi\in B^{\otimes\ell}(\varkappa),{\text{ }}\mu+\pi
(t)\in{\mathcal{C}}{\text{ for }}t\in\lbrack0,\ell],\pi(\ell)=\lambda}a_{\pi}$
by an easy extension of Assertion 10 in Theorem \ref{Th_Littel}. We can now
establish the following theorem.

\begin{theorem}
The law of the random walk $W$ conditioned by the event
\[
E:=({\mathcal{W}}(t)\in{\mathcal{C}}\ {\text{ for any }}t\geq0)
\]
is the same as the law of the Markov chain $H$ defined as the generalized
Pitman transform of $W$ (see Theorem \ref{Th_LawH}). The associated transition
matrix $\Pi^{E}$ verifies
\begin{equation}
\Pi^{E}(\mu,\lambda)=\frac{S_{\lambda}(\tau)}{S_{\mu}(\tau)\Sigma_{M}(\tau
)}m_{M,\mu}^{\lambda}\tau^{\mu-\lambda} \label{PitDEc}%
\end{equation}
and we have yet
\[
{\mathbb{P}}_{\mu}({\mathcal{W}}(t)\in{\mathcal{C}}\ {\text{ for any }}%
t\geq0)=\prod_{\alpha\in R_{+}}(1-\tau^{\alpha})S_{\mu}(\tau).
\]

\end{theorem}

\begin{proof}
The computation of the harmonic function $\psi(\mu)={\mathbb{P}}_{\mu
}({\mathcal{W}}(t)\in{\mathcal{C}}$ for any $t\geq0)$ is similar to
\S \ \ref{Subsec_Harm}.\ We have from the Weyl character formula
\[
\prod_{\alpha\in R_{+}}(1-e^{-\alpha})e^{-\mu}s_{\mu}=\sum_{\lambda\in P_{+}%
}f_{\lambda/\mu}^{\ell}\sum_{w\in{\mathsf{W}}}\varepsilon(w)\frac
{e^{w(\lambda+\rho)-\rho-\mu}}{s_{M}^{\ell}}
\]
where $s_{M}={\mathrm{char}}(M)$. When we specialize $\tau_{i}=e^{-\alpha_{i}%
}$ in $s_{M}$, we obtain $\Sigma_{M}(\tau)$.\ Hence
\[
\prod_{\alpha\in R_{+}}(1-\tau^{\alpha})S_{\mu}(\tau)=\sum_{\lambda\in P_{+}%
}f_{\lambda/\mu}^{\ell}\sum_{w\in{\mathsf{W}}}\varepsilon(w)\frac{\tau
^{\mu+\rho-w(\lambda+\rho)}}{\Sigma_{M}(\tau)^{\ell}}.
\]
If we set $\displaystyle \Pi_{\ell}^{w}(\mu):=\sum_{\lambda\in P_{+}%
}f_{\lambda/\mu}^{\ell}\frac{\tau^{\rho+\mu-w(\lambda+\rho)}}{\Sigma_{M}%
(\tau)^{\ell}}$, we yet obtain $\lim_{\ell\rightarrow+\infty}\Pi_{\ell}%
^{w}(\mu)=0$ when $w\neq1$ and $\Pi_{\ell}^{1}(\mu)=\psi_{\ell}(\mu
)$.\ Moreover
\[
\prod_{\alpha\in R_{+}}(1-\tau^{\alpha})S_{\mu}(\tau)=\sum_{w\in{\mathsf{W}}%
}\varepsilon(w)\Pi_{\ell}^{w}(\mu)=\Pi_{\ell}^{1}(\mu)+\sum_{w\neq
1}\varepsilon(w)\Pi_{\ell}^{w}(\mu)
\]
so the harmonic function $\psi=\lim_{\ell\rightarrow+\infty}\psi_{\ell}$ is
also given by $\psi(\mu)=\prod_{\alpha\in R_{+}}(1-\tau^{\alpha})S_{\mu}%
(\tau)$. Since $\Pi^{E}$ is the Doob $\psi$-transform of the the restriction
(\ref{SubstochM}) of $W$ to ${\mathcal{C}}$, we obtain the desired expression
(\ref{PitDEc}) for $\Pi^{E}(\mu,\lambda)$.

To see that $\Pi^{E}$ coincides with the law of the image of $W$ under the
generalized Pitman transform, we proceed as in Proof of Theorem \ref{Th_LawH}.
Consider $\mu=\mu^{(\ell)},\mu^{(\ell-1)},\ldots,\mu^{(1)}$ a sequence of
elements in $P_{+}$.\ Let ${\mathcal{S}}(\mu^{(1)},\ldots\mu^{(\ell)}%
,\lambda)$ be the set of paths $b^{h}\in B^{\otimes\ell}(\varkappa)$ remaining
in ${\mathcal{C}}$ and such that $b^{h}(k)=\mu^{(k)},k=1,\ldots,\ell$ and
$b^{(\ell+1)}=\lambda$. Consider $b=b_{1}\otimes\cdots\otimes b_{\ell}\otimes
b_{\ell+1}\in B^{\otimes\ell+1}(\varkappa)$. We have ${\mathfrak{P}}%
(b_{1}\otimes\cdots\otimes b_{k})(k)=\mu^{(k)}$ for any $k=1,\ldots,\ell$ and
${\mathfrak{P}}(b)(\ell+1)=\lambda$ if and only if ${\mathfrak{P}}%
(b)\in{\mathcal{S}}(\mu^{(1)},\ldots\mu^{(\ell)},\lambda)$. Moreover, by
(\ref{potimesell}), for any $b^{h}\in S(\mu^{(1)},\ldots\mu^{(\ell)},\lambda
)$, we have
\[
{\mathbb{P}}(b\in B(b^{h}))=\sum_{b\in B(b^{h})}p_{b}=\sum_{b\in B(b^{h}%
)}a_{b}\frac{\tau^{-{\mathrm{wt}}(b)}}{\Sigma_{M}(\tau)^{\ell+1}}=a_{b^{h}%
}\frac{\tau^{-\lambda}S_{\lambda}(\tau)}{\Sigma_{M}(\tau)^{\ell+1}}
\]
since $a_{b}=a_{b_{h}}$ for any $b\in B(b^{h}).$ This gives\
\begin{multline*}
{\mathbb{P}}(H_{\ell+1}=\lambda,H_{k}=\mu^{(k)},\forall k=1,\ldots,\ell
)=\frac{\tau^{-\lambda}S_{\lambda}(\tau)}{\Sigma_{M}(\tau)^{\ell+1}}%
\sum_{b^{h}\in{\mathcal{S}}(\mu^{(1)},\ldots\mu^{(\ell)},\lambda)}a_{b^{h}}\\
=\frac{\tau^{-\lambda}S_{\lambda}(\tau)}{\Sigma_{M}(\tau)^{\ell+1}}\prod
_{k=1}^{\ell-1}m_{\mu^{(k)},M}^{\mu^{(k+1)}}\times m_{\mu,M}^{\lambda}%
\end{multline*}
also using extension of Assertion 10 in Theorem \ref{Th_Littel}.\ Similarly
\[
{\mathbb{P}}(H_{k}=\mu^{(k)},\forall k=1,\ldots,\ell)=\frac{\tau^{-\mu}S_{\mu
}(\tau)}{\Sigma_{M}(\tau)^{\ell}}\prod_{k=1}^{\ell-1}m_{\mu^{(k)},M}%
^{\mu^{(k+1)}}
\]
which implies
\[
{\mathbb{P}}(H_{\ell+1}=\lambda\mid H_{k}=\mu^{(k)},\forall k=1,\ldots
,\ell)=\frac{S_{\lambda}(\tau)}{S_{\mu}(\tau)\Sigma_{M}(\tau)}m_{M,\mu
}^{\lambda}\tau^{\mu-\lambda}.
\]

\end{proof}

\subsection{Example: random walk to the height closest neighbors}

We now study in detail the case of the random walk in the plane with
transitions $0$ and the height closest neighbors. The underlying
representation is not irreducible and does not decompose as a sum of minuscule
representations. So the conditioning of this walk cannot be obtained by the
methods of \cite{LLP}.\footnote{The results of \cite{LLP} permit nevertheless
to study the random walk \emph{in the space} ${\mathbb{R}}^{3}$ with
transitions $\pm\varepsilon_{1}\pm\varepsilon_{2}\pm\varepsilon_{3}$
corresponding to the weights of the spin representation of ${\mathfrak{g}%
}={\mathfrak{so}}_{9}$.}

The root system of type $C_{2}$ is realized in ${\mathbb{R}}^{2}%
{\mathbb{=R\varepsilon}}_{1}\oplus{\mathbb{R\varepsilon}}_{2}$. The Cartan
matrix is
\[
A=\left(
\begin{array}
[c]{cc}%
2 & -1\\
-2 & 2
\end{array}
\right)
\]
The simple roots are then $\alpha_{1}=\varepsilon_{1}-\varepsilon_{2}$ and
$\alpha_{2}=2\varepsilon_{2}$. We have $P={\mathbb{Z}}^{2}$.\ The fundamental
weights are $\omega_{1}=\varepsilon_{1}$ and $\omega_{2}=\varepsilon
_{1}+\varepsilon_{2}.$ We have ${\mathcal{C}}=\{(x,y)\in{\mathbb{R}}^{2}\mid
x\geq y\geq0\}$ and $P_{+}=\{\lambda=(\lambda_{1},\lambda_{2})\mid\lambda
_{1}\geq\lambda_{2}\geq0\},$ the set of partitions with two parts . Choose
$\tau_{1}\in]0,1[,\tau_{2}\in]0,1[.\ $For $\lambda=(\lambda_{1},\lambda
_{2})\in P_{+}$, we have $\lambda=\lambda_{1}\alpha_{1}+\frac{\lambda
_{1}+\lambda_{2}}{2}\alpha_{2}$. Thus $\tau^{\lambda}=\tau_{1}^{\lambda_{1}%
}(\sqrt{t_{2}})^{\lambda_{1}+\lambda_{2}}$.\ 

Consider the ${\mathfrak{sp}}_{4}({\mathbb{C}})$-module $M=V(1)^{\oplus a_{1}%
}\oplus V(1,1)^{\oplus a_{2}}$. The elementary paths in $B(\varkappa)$ can be
easily described from the highest weight paths
\[
\pi_{1}:t\mapsto t\varepsilon_{1}{\text{ and }}\gamma_{12}:\left\{
\begin{array}
[c]{l}%
2t\varepsilon_{1},t\in\lbrack0,\frac{1}{2}]\\
\varepsilon_{1}+2(t-\frac{1}{2})\varepsilon_{2},t\in]\frac{1}{2},1]
\end{array}
\right.  {\text{ in }}{\mathcal{C}}{\text{.}}
\]
We obtain $B(\varkappa)=B(\pi_{1})\oplus B(\gamma_{12})$ where

\begin{enumerate}
\item $B(\pi_{1}):\pi_{1}:t\mapsto t\varepsilon_{1},\pi_{2}:t\mapsto
t\varepsilon_{2},\pi_{\overline{1}}:t\mapsto-t\varepsilon_{1}$ and
$\pi_{\overline{2}}:t\mapsto-t\varepsilon_{2}$ with $t\in\lbrack0,1]$

\item $B(\gamma_{12}):$
\begin{align*}
\gamma_{12}  &  :\left\{
\begin{array}
[c]{l}%
2t\varepsilon_{1},t\in\lbrack0,\frac{1}{2}]\\
\varepsilon_{1}+2(t-\frac{1}{2})\varepsilon_{2},t\in]\frac{1}{2},1]
\end{array}
\right.  \quad\gamma_{1\overline{2}}:\left\{
\begin{array}
[c]{l}%
2t\varepsilon_{1},t\in\lbrack0,\frac{1}{2}]\\
\varepsilon_{1}-2(t-\frac{1}{2})\varepsilon_{2},t\in]\frac{1}{2},1]
\end{array}
\right.  ,\\
\gamma_{2\overline{2}}  &  :\left\{
\begin{array}
[c]{l}%
2t\varepsilon_{2},t\in\lbrack0,\frac{1}{2}]\\
\varepsilon_{2}-2(t-\frac{1}{2})\varepsilon_{2},t\in]\frac{1}{2},1]
\end{array}
\right. \\
\gamma_{2\overline{1}}  &  :\left\{
\begin{array}
[c]{l}%
2t\varepsilon_{2},t\in\lbrack0,\frac{1}{2}]\\
\varepsilon_{2}-2(t-\frac{1}{2})\varepsilon_{1},t\in]\frac{1}{2},1]
\end{array}
\right.  {\text{ and }}\quad\gamma_{\overline{2}\overline{1}}:\left\{
\begin{array}
[c]{l}%
-2t\varepsilon_{2},t\in\lbrack0,\frac{1}{2}]\\
-\varepsilon_{2}-2(t-\frac{1}{2})\varepsilon_{1},t\in]\frac{1}{2},1]
\end{array}
\right.  .
\end{align*}

\end{enumerate}

The crystal $B(\varkappa)$ is the union of the two following crystals
\begin{gather*}
\pi_{1}\overset{1}{\rightarrow}\pi_{2}\overset{2}{\rightarrow}\pi
_{\overline{2}}\overset{1}{\rightarrow}\pi_{\overline{1}}\\
\gamma_{12}\overset{2}{\rightarrow}\gamma_{1\overline{2}}\overset
{1}{\rightarrow}\gamma_{2\overline{2}}\overset{1}{\rightarrow}\gamma
_{2\overline{1}}\overset{2}{\rightarrow}\gamma_{\overline{2}\overline{1}}%
\end{gather*}
Observe that for the path $\gamma_{2\overline{2}}$, we have $\gamma
_{2\overline{2}}(0)=\gamma_{2\overline{2}}(1)=0$. The other transitions
correspond to the 8 closest neighbors in the lattice ${\mathbb{Z}}^{2}$.
\bigskip

We now define the probability distribution $p$ on the set $B(\pi_{1})^{\oplus
m_{1}}\oplus B(\gamma_{12})^{\oplus m_{2}}$. We have
\[
\Sigma_{M}(\tau)=a_{1}\frac{1+\tau_{1}+\tau_{1}\tau_{2}+\tau_{1}^{2}\tau_{2}%
}{\tau_{1}\sqrt{\tau_{2}}}+a_{2}\frac{1+\tau_{2}+\tau_{1}\tau_{2}+\tau_{1}%
^{2}\tau_{2}+\tau_{1}^{2}\tau_{2}^{2}}{\tau_{1}\tau_{2}}.
\]
The probability $p$ is defined by
\begin{align*}
p_{1}  &  =\frac{a_{1}}{\Sigma_{M}(\tau)\tau_{1}\sqrt{\tau_{2}}},\quad
p_{2}=\frac{a_{1}}{\Sigma_{M}(\tau)\sqrt{\tau_{2}}},\quad p_{\overline{2}%
}=\frac{a_{1}\sqrt{\tau_{2}}}{\Sigma_{M}(\tau)},\quad p_{\overline{1}}%
=\frac{a_{1}\tau_{1}\sqrt{\tau_{2}}}{\Sigma_{M}(\tau)}\\
p_{12}  &  =\frac{a_{2}}{\Sigma_{M}(\tau)\tau_{1}\tau_{2}},\quad
p_{1\overline{2}}=\frac{a_{2}}{\Sigma_{M}(\tau)\tau_{1}},\quad p_{2\overline
{2}}=\frac{a_{2}}{\Sigma_{M}(\tau)},\quad p_{2\overline{1}}=\frac{a_{2}%
\tau_{1}}{\Sigma_{M}(\tau)},\quad p_{\overline{2}\overline{1}}=\frac{a_{2}%
\tau_{1}\tau_{2}}{\Sigma_{M}(\tau)}.
\end{align*}
The set of positive roots is
\[
R_{+}=\{\varepsilon_{1}\pm\varepsilon_{2},2\varepsilon_{1},2\varepsilon
_{2}\}{\text{ and }}\rho=(2,1).
\]
The action of the Weyl group on ${\mathbb{Z}}^{2}$ yields the 8
transformations which preserves the square of vertices $(\pm1,\pm1)$. For any
partition $\mu=(\mu_{1},\mu_{2})\in P_{+}$, we obtain by the Weyl character
formula and Theorem \ref{Th_PSi}
\begin{multline*}
\psi(\mu)={\mathbb{P}}_{\mu}({\mathcal{W}}(t)\in{\mathcal{C}},t\geq
0)=(1-\tau_{1})(1-\tau_{2})(1-\tau_{1}\tau_{2})(1-\tau_{1}^{2}\tau_{2})S_{\mu
}(\tau_{1},\tau_{2})=\\
\sum_{w\in W}\varepsilon(w)\tau^{w(\mu+\rho)-(\mu+\rho)}=\\
1+\tau_{1}^{\mu_{1}-\mu_{2}+1}\tau_{2}^{\mu_{1}+2}+\tau_{1}^{2\mu_{1}+4}%
\tau_{2}^{\mu_{1}+\mu_{2}+3}+\tau_{1}^{\mu_{1}+\mu_{2}+3}\tau_{2}^{\mu_{2}%
+1}\\
-\tau_{1}^{\mu_{1}-\mu_{2}+1}-\tau_{2}^{\mu_{2}+1}-\tau_{1}^{2\mu_{1}+4}%
\tau_{2}^{\mu_{1}+2}-\tau_{1}^{\mu_{1}+\mu_{2}+3}\tau_{2}^{\mu_{1}+\mu_{2}+3}%
\end{multline*}
Moreover, the law of the random walk $W$ conditioned by the event
\[
E:=({\mathcal{W}}(t)\in{\mathcal{C}}\ {\text{ for any }}t\geq0)
\]
is the same as the law of the Markov chain $H$ defined as the generalized
Pitman transform of $W$ (see Theorem \ref{Th_LawH}). To compute the associated
transition matrix $M,$ we need the tensor product multiplicities $m_{\mu
,M}^{\lambda}=a_{1}m_{(1,0),\mu}^{\lambda}+a_{2}m_{(1,1),\mu}^{\lambda}$. We
have for any partitions $\lambda$ and $\mu$ with two parts
\[
m_{(1,0),\mu}^{\lambda}=\left\{
\begin{array}
[c]{l}%
1{\text{ if }}\lambda{\text{ and }}\mu{\text{ are equal or differ by only one
box}}\\
0{\text{ otherwise}}%
\end{array}
\right.
\]
and
\[
m_{(1,1),\mu}^{\lambda}=\left\{
\begin{array}
[c]{l}%
1{\text{ if }}\lambda{\text{ and }}\mu{\text{ are equal or differ by two boxes
in different rows}}\\
0{\text{ otherwise.}}%
\end{array}
\right.
\]
We thus have for any $\lambda,\mu\in P_{+}$
\[
\Pi^{E}(\mu,\lambda)=\frac{\psi(\lambda)}{\psi(\mu)\Sigma_{M}(\tau)}\left(
a_{1}m_{(1,0),\mu}^{\lambda}+a_{2}m_{(1,1),\mu}^{\lambda}\right)  \tau
_{1}^{\mu_{1}-\lambda_{1}}\sqrt{\tau_{2}}^{(\mu_{1}+\mu_{2}-\lambda
_{1}-\lambda_{2})}.
\]

\section{Some consequences}

{In the remaining of the paper, \emph{we assume that }$g$\emph{ is of finite
type} and ${\mathcal{W}}$ is constructed from an irreducible ${\mathfrak{g}}%
$-module $V(\kappa)$ in the category ${\mathcal{O}}_{int}$. Then the crystal
}$B(\pi_{\kappa})$ has a finite number of paths which all have the same length
as $\pi_{\kappa}$ since $W$ contains only isometries.

\subsection{Asymptotics for the multiplicities $f_{\lambda/\mu}^{\ell}$}

We will use later a quotient version of a local limit theorem for these random
paths; following \cite{LLP}, we may state the

\begin{proposition}
\label{quotientLLT} Let $(g_{\ell}),(h_{\ell})$ be two sequences in $P$ such
that the events $(W_{\ell}=g_{\ell})$ and $(W_{\ell}=g_{\ell}+h_{\ell})$ have
non zero probability for $\ell>0$ large enough. Assume there exists
$\alpha<2/3$ such that $\lim{\ell}^{-\alpha}\|g_{\ell}-{\ell}m\|=0$ and
$\lim{\ell}^{-1/2}\|h_{\ell}\|=0$. Then, when ${\ell}$ tends to infinity, one
gets
\[
{\mathbb{P}}_{0}(W_{\ell}=g_{\ell}+h_{\ell},{\mathcal{W}}(t)\in{\mathcal{C}%
}\ {\text{ for any }}t\in\lbrack0,\ell])\sim{\mathbb{P}}_{0}(W_{\ell}=g_{\ell
},{\mathcal{W}}(t)\in{\mathcal{C}}\ {\text{ for any }}t\in\lbrack0,\ell]).
\]

\end{proposition}

\begin{proof}
The proof of this statement follows line by line the one of Theorem 4.3 in
\cite{LLP}. Without loss of generality, we may assume that the law of the
$X_{\ell}$ is aperiodic in $P$, which means that its support generates $P$ and
is not included in a coset of a proper subgroup of $P$: this readily implies
that ${\mathbb{P}}_{0}(W_{\ell}=g_{\ell})>0$ and ${\mathbb{P}}_{0}(W_{\ell
}=g_{\ell}+h_{\ell})>0$ for when $(g_{\ell})_{\ell}$ and $(h_{\ell})_{\ell}$
satisfy the conditions of the proposition and $\ell$ large enough. When the
law of the $X_{\ell}$ is not aperiodic, the random walk $({\mathcal{W}}_{\ell
})_{\ell}$ has a finite number $p$ of periodic classes and the condition
${\mathbb{P}}_{0}(W_{\ell}=g_{\ell})>0$ and ${\mathbb{P}}_{0}(W_{\ell}%
=g_{\ell}+h_{\ell})>0$ corresponds to the fact that $g_{\ell}$ and $g_{\ell
}+h_{\ell}$ belong to the same periodic class indexed by the value of $\ell$
modulo $p$; the statement in this case follows from the one in the aperiodic
one, by induction of the random walk on each periodic class.

We fix a real number $\beta$ such that $\frac{1}{2}<\alpha<\beta<\frac{2}{3}$,
set $b_{\ell}=[{\ell}^{\beta}]$ and choose $\delta>0$ be such that $B_{{\ell}%
}=:=B(m,\delta)\subset{\mathcal{C}}$.

As in \cite{LLP}, we first check that $\displaystyle\frac{{\mathbb{P}}%
_{0}(W_{\ell}=g_{\ell},{\mathcal{W}}(t)\in{\mathcal{C}}\ {\text{ for any }%
}t\in\lbrack0,\ell])}{{\mathbb{P}}_{0}(W_{\ell}=g_{\ell},W_{b_{\ell}}\in
B_{b_{\ell}},{\mathcal{W}}(t)\in{\mathcal{C}}\ {\text{ for any }}t\in
\lbrack0,b_{\ell}])}\rightarrow1;$ in others words, one may \textquotedblleft
forget\textquotedblright\ the conditioning $({\mathcal{W}}(t)\in{\mathcal{C}%
}\ {\text{ for any }}t\in\lbrack b_{\ell},l])$ in the event $(W_{\ell}%
=g_{\ell},{\mathcal{W}}(t)\in{\mathcal{C}}\ {\text{ for any }}t\in
\lbrack0,\ell])$. The same result holds if one replaces $g_{\ell}$ by
$g_{\ell}+h_{\ell}$ for $\lim{\ell}^{-\alpha}\Vert g_{\ell}+h_{\ell}-{\ell
}m\Vert=0$.

To achieve the proof of the proposition, it now suffices to establish that
\[
\frac{{\mathbb{P}}_{0}(W_{\ell}=g_{\ell}+h_{\ell},W_{b_{\ell}}\in B_{b_{\ell}%
},{\mathcal{W}}(t)\in{\mathcal{C}}\ {\text{ for any }}t\in\lbrack0,b_{\ell}%
])}{{\mathbb{P}}_{0}(W_{\ell}=g_{\ell},W_{b_{\ell}}\in B_{b_{\ell}%
},{\mathcal{W}}(t)\in{\mathcal{C}}\ {\text{ for any }}t\in\lbrack0,b_{\ell}%
])}\ \rightarrow1.
\]
Since the increments of the random walk $({\mathcal{W}}_{\ell})_{\ell}$ are
independent with the same law, we may write
\begin{multline*}
{\mathbb{P}}_{0}(W_{\ell}=g_{\ell},W_{b_{\ell}}\in B_{b_{\ell}},{\mathcal{W}%
}(t)\in{\mathcal{C}}\ {\text{ for any }}t\in\lbrack0,b_{\ell}])\\
=\sum_{x\in B_{b_{\ell}}\cap P_{+}}{\mathbb{P}}_{0}({\mathcal{W}}_{{\ell
}-b_{\ell}}=g_{\ell}-x)\times{\mathbb{P}}_{0}({\mathcal{W}}_{\ell
}=x,{\mathcal{W}}(t)\in{\mathcal{C}}\ {\text{ for any }}t\in\lbrack0,b_{\ell
}]).
\end{multline*}
This leads to the proposition since $\displaystyle{\mathbb{P}}_{0}%
({\mathcal{W}}_{{\ell}-b_{\ell}}=g_{\ell}-x)\sim{\mathbb{P}}_{0}({\mathcal{W}%
}_{{\ell}-b_{\ell}}=g_{\ell}+h_{\ell}-x)$ uniformly in $x\in B_{b_{\ell}}$.
\end{proof}

\bigskip

Consider $\lambda,\mu\in P_{+}$ and $\ell\geq1$ such that $f_{\lambda/\mu
}^{\ell}>0$ and $f_{\lambda}^{\ell}>0$. Then, we must have $\ell\kappa
+\mu-\lambda\in Q_{+}$ and $\ell\kappa-\lambda\in Q_{+}$. Therefore $\mu\in Q$
and it decomposes as a sum of simple roots. In the sequel, we will assume the
condition $\mu\in Q\cap P_{+}$ is satisfied.

We assume the notation and hypotheses of Theorem \ref{Th_PSi}. Consider a
sequence $\lambda^{(\ell)}$ of dominant weights such that $\lambda^{(\ell
)}=\ell m(1)+o(\ell)$. Following Proposition 5.3 in \cite{LLP}, one gets the
following decomposition
\begin{equation}
\frac{f_{\lambda^{(\ell)}/\mu}^{\ell}}{f_{\lambda^{(\ell)}}^{\ell}}%
=\sum_{\gamma\in P}K_{\mu,\gamma}\frac{f_{\lambda^{(\ell)}-\gamma}^{\ell}%
}{f_{\lambda^{(\ell)}}^{\ell}}=\tau^{-\mu}\sum_{\gamma\in P}K_{\mu,\gamma
}\frac{f_{\lambda^{(\ell)}-\gamma}^{\ell}\tau^{-\lambda^{(\ell)}+\gamma}%
}{f_{\lambda^{(\ell)}}^{\ell}\tau^{-\lambda^{(\ell)}}}\tau^{\mu-\gamma}
\label{quo-f}%
\end{equation}
where the sums are finite since the set of weights in $V(\mu)$ is finite. By
Assertion 2 of Proposition \ref{Prop_util}, we have, for any $\gamma\in P$ and
$\ell$ large enough
\[
\frac{f_{\lambda^{(\ell)}-\gamma}^{\ell}\tau^{-\lambda^{(\ell)}+\gamma}%
}{f_{\lambda^{(\ell)}}^{\ell}\tau^{-\lambda^{(\ell)}}}=\frac{f_{\lambda
^{(\ell)}-\gamma}^{\ell}\tau^{\ell\kappa-\lambda^{(\ell)}+\gamma}}%
{f_{\lambda^{(\ell)}}^{\ell}\tau^{\ell\kappa-\lambda^{(\ell)}}}=\frac
{{\mathbb{P}}({\mathcal{W}}_{\ell}=\lambda^{(\ell)}-\gamma,{\mathcal{W}}%
_{t}\in{\mathcal{C}}{\text{ for any }}t\in\lbrack0,\ell])}{{\mathbb{P}%
}({\mathcal{W}}_{\ell}=\lambda^{(\ell)},{\mathcal{W}}_{t}\in{\mathcal{C}%
}{\text{ for any }}t\in\lbrack0,\ell])}%
\]
this last quotient tending to $1$ when $\ell$ tends to infinity, by
Proposition \ref{quotientLLT}. This implies
\[
\lim_{\ell\rightarrow+\infty}\frac{f_{\lambda^{(\ell)}/\mu}^{\ell}}%
{f_{\lambda^{(\ell)}}^{\ell}}=\tau^{-\mu}\sum_{\gamma\in P}K_{\mu,\gamma}%
\tau^{\mu-\gamma}=\tau^{-\mu}S_{\mu}(\tau).
\]
We have thus proved the following consequence of Theorem \ref{Th_PSi}

\begin{corollary}
For any $\mu\in Q\cap P_{+}$, and any sequence of dominant weights of the form
$\lambda^{(\ell)}=\ell m(1)+o(\ell)$, we have $\lim_{\ell\rightarrow+\infty
}\frac{f_{\lambda^{(\ell)}/\mu}^{\ell}}{f_{\lambda}^{\ell}}=\tau^{-\mu}S_{\mu
}(\tau)$.
\end{corollary}

\noindent\textbf{Remark: }One can regard this corollary as an analogue of the
asymptotic behavior of the number of paths in the Young lattice obtained by
Kerov and Vershik (see \cite{Ker} and the references therein).

\subsection{Probability that $W$ stay in ${\mathcal{C}}$}

By Theorem \ref{Th_PSi}, we can compute ${\mathbb{P}}_{\mu}({\mathcal{W}%
}(t)\in{\mathcal{C}}\ {\text{ for any }}t\in\lbrack0,\ell])$.\ Unfortunately,
this does not permit do make explicit ${\mathbb{P}}_{\mu}(W_{\ell}%
\in{\mathcal{C}}$ $\forall\ell\geq1)$. Nevertheless, we have the immediate
inequality
\[
{\mathbb{P}}_{\mu}({\mathcal{W}}(t)\in{\mathcal{C}}\ {\text{ for any }}%
t\geq0)\leq{\mathbb{P}}_{\mu}(W_{\ell}\in{\mathcal{C}}\ {\text{ for any }}%
\ell\geq1).
\]
Since we have assumed that ${\mathfrak{g}}$ is of finite type, each crystal
$B(\pi_{\kappa})$ is finite. For any $i=1,\ldots,n$, write $m_{0}(i)\geq1$ for
the maximal length of the $i$-chains appearing in $B(\pi_{\kappa})$.\ Set
$\kappa_{0}=\sum_{i=1}^{n}(m_{0}(i)-1)\omega_{i}$.\ Observe that $\kappa
_{0}=0$ if and only if $\kappa$ is a minuscule weight.

\begin{lemma}
Assume $W_{k}\in{\mathcal{C}}$ for any $k=1,\ldots,\ell$. Then $\kappa
_{0}+{\mathcal{W}}(t)\in{\mathcal{C}}$ for any $t\in\lbrack0,\ell].$
\end{lemma}

\begin{proof}
Since $\kappa_{0}$ is a dominant weight, we can consider $\pi_{\kappa_{0}}$
any path from $0$ to $\kappa_{0}$ which remains in ${\mathcal{C}}$.\ First
observe that the hypothesis $W_{k}\in{\mathcal{C}}$ for any $k=1,\ldots,\ell$
is equivalent to $\kappa_{0}+{\mathcal{W}}(k)\in\kappa_{0}+{\mathcal{C}}$ for
any $k=1,\ldots,\ell.$ We also know by Assertion 8 of Theorem \ref{Th_Littel}
that $B(\pi_{\kappa_{0}})\otimes B(\pi_{\kappa})^{\otimes\ell}$ is contained
in ${\mathcal{P}}_{\min{\mathbb{Z}}}$ for any $\ell\geq1$. Set ${\mathcal{W}%
}(\ell)=\pi_{1}\otimes\cdots\otimes\pi_{\ell}$. By Assertion 3 of Proposition
\ref{prop-ac_etil}, we have to prove that $\tilde{e}_{i}(\pi_{\kappa_{0}%
}\otimes\pi_{1}\otimes\cdots\otimes\pi_{\ell})=0$ for any $i=1,\ldots,n$
providing ${\mathrm{wt}}(\pi_{\kappa_{0}}\otimes\pi_{1}\otimes\cdots\otimes
\pi_{k})=\pi_{\kappa_{0}}\otimes\pi_{1}\otimes\cdots\otimes\pi_{k}(1)\in
\kappa_{0}+P_{+}$ for any $k=1,\ldots,\ell$. Fix $i=1,\ldots,n$. Set
$\kappa_{0}(i)=m_{0}(i)-1.\ $We proceed by induction.

Assume $\ell=1$.\ Since we have $\tilde{e}_{i}(\pi_{\kappa_{0}})=0$, it
suffices to prove by using Assertion 2 of Proposition \ref{Prop_HP} that
$\varepsilon_{i}(\pi_{1})\leq\varphi_{i}(\pi_{\kappa_{0}})$. By definition of
the dominant weight $\pi_{\kappa_{0}}$, we have $\varphi_{i}(\pi_{\kappa_{0}%
})=\kappa_{0}(i)$. So we have to prove that $\varepsilon_{i}(\pi_{1}%
)\leq\kappa_{0}(i)$.\ Assertion 7 of Theorem \ref{Th_Littel} and the
hypothesis ${\mathrm{wt}}(\pi_{\kappa_{0}}\otimes\pi_{1})\in\kappa_{0}+P_{+}$
permits to write
\begin{equation}
{\mathrm{wt}}(\pi_{\kappa_{0}})_{i}+{\mathrm{wt}}(\pi_{1})_{i}={\mathrm{wt}%
}(\pi_{\kappa_{0}}\otimes\pi_{1})_{i}\geq\kappa_{0}(i). \label{weights}%
\end{equation}
Recall that $\pi_{1}$ belongs to $B(\pi_{\kappa})$. So $\varepsilon_{i}%
(\pi_{1})\leq\kappa_{0}(i)+1$ because $\varepsilon_{i}(\pi_{1})$ gives the
distance of $\pi_{1}$ from the source vertex of its $i$-chain. When
$\varepsilon_{i}(\pi_{1})<\kappa_{0}(i)+1$ we are done. So assume
$\varepsilon_{i}(\pi_{1})=\kappa_{0}(i)+1$. This means that $\pi_{1}$
satisfies $\varphi_{i}(\pi_{1})=0$. Therefore, ${\mathrm{wt}}(\pi_{1}%
)_{i}=-\kappa_{0}(i)-1$. But in this case, we get by (\ref{weights})
\[
{\mathrm{wt}}(\pi_{\kappa_{0}}\otimes\pi_{1})_{i}=\kappa_{0}(i)-(\kappa
_{0}(i)+1)=-1\geq\kappa_{0}(i)
\]
hence a contradiction.

Now assume $\tilde{e}_{i}(\pi_{\kappa_{0}}\otimes\pi_{1}\otimes\cdots
\otimes\pi_{\ell-1})=0$ for any $k=1,\ldots,\ell-1$. Observe that
${\mathrm{wt}}(\pi_{\kappa_{0}}\otimes\pi_{1}\otimes\cdots\otimes\pi_{\ell
-1})_{i}=\varphi_{i}(\pi_{\kappa_{0}}\otimes\pi_{1}\otimes\cdots\otimes
\pi_{\ell-1})\geq\kappa_{0}(i)$ since $\pi_{\kappa_{0}}\otimes\pi_{1}%
\otimes\cdots\otimes\pi_{\ell-1}\in\kappa_{0}+P_{+}$ and $\tilde{e}_{i}%
(\pi_{\kappa_{0}}\otimes\pi_{1}\otimes\cdots\otimes\pi_{\ell-1})=0$. We also
have
\begin{equation}
{\mathrm{wt}}(\pi_{\kappa_{0}}\otimes\pi_{1}\otimes\cdots\otimes\pi_{\ell
-1}\otimes\pi_{\ell})_{i}={\mathrm{wt}}(\pi_{\kappa_{0}}\otimes\pi_{1}%
\otimes\cdots\otimes\pi_{\ell-1})_{i}+{\mathrm{wt}}(\pi_{\ell})_{i}\geq
\kappa_{0}(i). \label{weighs2}%
\end{equation}

We proceed as in the case $\ell=1.$ Assume first $\varepsilon_{i}(\pi_{\ell
})\leq\varphi_{i}(\pi_{\kappa_{0}}\otimes\pi_{1}\otimes\cdots\otimes\pi
_{\ell-1})$. Then by Proposition \ref{Prop_HP} and the induction equality
$\tilde{e}_{i}(\pi_{\kappa_{0}}\otimes\pi_{1}\otimes\cdots\otimes\pi_{\ell
-1})=0$, we will have $\tilde{e}_{i}(\pi_{\kappa_{0}}\otimes\pi_{1}%
\otimes\cdots\otimes\pi_{\ell})=0$.

Now assume $\varepsilon_{i}(\pi_{\ell})>\varphi_{i}(\pi_{\kappa_{0}}\otimes
\pi_{1}\otimes\cdots\otimes\pi_{\ell-1})$. Since $\varphi_{i}(\pi_{\kappa_{0}%
}\otimes\pi_{1}\otimes\cdots\otimes\pi_{\ell-1})\geq\kappa_{0}(i)$ and
$\pi_{\ell}\in B(\pi_{\kappa_{0}})$, we must have $\varepsilon_{i}(\pi_{\ell
})=\kappa_{0}(i)+1,\varphi_{i}(\pi_{\ell})=0$ and $\varphi_{i}(\pi_{\kappa
_{0}}\otimes\pi_{1}\otimes\cdots\otimes\pi_{\ell-1})=\kappa_{0}(i)$.
Therefore, we get ${\mathrm{wt}}(\pi_{\ell})_{i}=-\kappa_{0}(i)-1$ and
${\mathrm{wt}}(\pi_{\kappa_{0}}\otimes\pi_{1}\otimes\cdots\otimes\pi_{\ell
-1})_{i}=\kappa_{0}(i)$.\ Then (\ref{weighs2}) yields yet the contradiction
\[
-1\geq\kappa_{0}(i).
\]

\end{proof}

\bigskip

\noindent\textbf{Remark:} In general the assertion $W_{k}\in{\mathcal{C}}$ for
any $k=1,\ldots,\ell$ is not equivalent to the assertion $\kappa
_{0}+{\mathcal{W}}(t)\in\kappa_{0}+{\mathcal{C}}$ for any $t\in\lbrack
0,\ell].$ This is nevertheless true when $\kappa$ is a minuscule weight since
$\kappa_{0}=0$ in this case and the paths in $B(\pi_{\kappa})$ are lines.
\bigskip

We deduce from the previous lemma the inequality
\[
{\mathbb{P}}_{\mu}(W_{k}\in{\mathcal{C}}\ {\text{ for any }}k=0,\ldots
,\ell)\leq{\mathbb{P}}_{\mu+\kappa_{0}}({\mathcal{W}}(t)\in{\mathcal{C}%
}\ {\text{ for any }}t\in\lbrack0,\ell]).
\]
When $\ell$ tends to infinity, this yields
\[
{\mathbb{P}}_{\mu}(W_{\ell}\in{\mathcal{C}}\ {\text{ for any }}\ell\geq
1)\leq{\mathbb{P}}_{\mu+\kappa_{0}}({\mathcal{W}}(t)\in{\mathcal{C}}\ {\text{
for any }}t\geq0.
\]
By using Theorem \ref{Th_PSi}, this implies the

\begin{theorem}
Assume ${\mathfrak{g}}$ is of finite type (then $m_{\alpha}=1$ for any
$\alpha\in R_{+}$). Then, for any $\mu\in P_{+}$ we have
\[
\prod_{\alpha\in R_{+}}(1-\tau^{\alpha})S_{\mu}(\tau)\leq{\mathbb{P}}_{\mu
}(W_{\ell}\in{\mathcal{C}}\ {\text{ for any }}\ell\geq1)\leq\prod_{\alpha\in
R_{+}}(1-\tau^{\alpha})S_{\mu+\kappa_{0}}(\tau).
\]
In particular, we recover the result of Corollary 7.4.3 in \cite{LLP} :
\[
{\mathbb{P}}_{\mu}(W_{\ell}\in{\mathcal{C}}\ {\text{ for any }}\ell
\geq1)=\prod_{\alpha\in R_{+}}(1-\tau^{\alpha})S_{\mu}(\tau)
\]
when $\kappa$ is minuscule.
\end{theorem}

\noindent\textbf{Remark:} The inequality obtained in the previous theorem can
also be rewritten
\[
1\leq\frac{{\mathbb{P}}_{\mu}(W_{\ell}\in{\mathcal{C}}\ {\text{ for any }}%
\ell\geq1)}{{\mathbb{P}}_{\mu}({\mathcal{W}}(t)\in{\mathcal{C}}\ {\text{ for
any }}t\in\lbrack0,+\infty\lbrack)}\leq\frac{S_{\mu+\kappa_{0}}(\tau)}{S_{\mu
}(\tau)}.
\]
When $\mu$ tends to infinity, we thus have ${\mathbb{P}}_{\mu}(W_{\ell}%
\in{\mathcal{C}}$ $\forall\ell\geq1)\sim{\mathbb{P}}_{\mu}({\mathcal{W}}%
(t)\in{\mathcal{C}}{\text{ for any }}t\geq0)$ as expected.

\section{Appendix (proof of Proposition \ref{Prop_Q})}

By definition of the probability ${\mathbb{Q}}$, for any $\ell\geq1$ and any
$\mu_{0},\cdots,\mu_{\ell},\lambda\in{\mathcal{C}}$, one gets
\begin{align*}
{\mathbb{Q}}(Y_{\ell+1}=\lambda\mid Y_{\ell}=\mu_{\ell},\cdots,Y_{0}=\mu_{0})
&  =\frac{{\mathbb{Q}}(Y_{\ell+1}=\lambda,Y_{\ell}=\mu_{\ell},\ldots,Y_{0}%
=\mu_{0})}{{\mathbb{Q}}(Y_{\ell}=\mu_{\ell},\ldots,Y_{0}=\mu_{0})}\\
&  =\frac{{\mathbb{P}}(E,Y_{\ell+1}=\lambda,Y_{\ell}=\mu_{\ell},\ldots
,Y_{0}=\mu_{0})}{{\mathbb{P}}(E,Y_{\ell}=\mu_{\ell},\ldots,Y_{0}=\mu_{0}%
)}=:{\frac{N_{\ell}}{D_{\ell}}}.
\end{align*}
We first have, using the Markov property
\begin{align*}
N_{\ell}  &  ={\mathbb{P}}({\mathcal{Y}}(t)\in{\mathcal{C}}{\text{ for }}%
t\geq1,Y_{\ell+1}=\lambda,Y_{\ell}=\mu_{\ell},\ldots,Y_{0}=\mu_{0})\\
&  ={\mathbb{P}}({\mathcal{Y}}(t)\in{\mathcal{C}}{\text{ for }}t\geq\ell+1\mid
Y_{\ell+1}=\lambda,{\mathcal{Y}}(t)\in{\mathcal{C}}{\text{ for }}t\in
\lbrack0,\ell+1],Y_{\ell}=\mu_{\ell},\ldots,Y_{0}=\mu_{0})\\
&  \ \qquad\qquad\qquad\times\quad{\mathbb{P}}(Y_{\ell+1}=\lambda
,{\mathcal{Y}}(t)\in{\mathcal{C}}{\text{ for }}t\in\lbrack0,\ell+1],Y_{\ell
}=\mu_{\ell},\ldots,Y_{0}=\mu_{0})\\
&  ={\mathbb{P}}({\mathcal{Y}}(t)\in{\mathcal{C}}{\text{ for }}t\geq\ell+1\mid
Y_{\ell+1}=\lambda)\\
&  \ \qquad\qquad\qquad\times\quad{\mathbb{P}}(Y_{\ell+1}=\lambda
,{\mathcal{Y}}(t)\in{\mathcal{C}}{\text{ for }}t\in\lbrack0,\ell+1[,Y_{\ell
}=\mu_{\ell},\ldots,Y_{0}=\mu_{0})\\
&  ={\mathbb{P}}({\mathcal{Y}}(t)\in{\mathcal{C}}{\text{ for }}t\geq0\mid
Y_{0}=\lambda)\\
&  \ \qquad\qquad\qquad\times\quad{\mathbb{P}}(Y_{\ell+1}=\lambda
,{\mathcal{Y}}(t)\in{\mathcal{C}}{\text{ for }}t\in\lbrack0,\ell+1[,Y_{\ell
}=\mu_{\ell},\ldots,Y_{0}=\mu_{0})
\end{align*}
with \newline${\mathbb{P}}(Y_{\ell+1}=\lambda,{\mathcal{Y}}(t)\in{\mathcal{C}%
}{\text{ for }}t\in\lbrack0,\ell+1[,Y_{\ell}=\mu_{\ell},\ldots,Y_{0}=\mu
_{0})$
\begin{align*}
\  &  ={\mathbb{P}}(Y_{\ell+1}=\lambda,{\mathcal{Y}}(t)\in{\mathcal{C}}{\text{
for }}t\in\lbrack\ell,\ell+1[\mid{\mathcal{Y}}(t)\in{\mathcal{C}}{\text{ for
}}t\in\lbrack0,\ell\lbrack,Y_{\ell}=\mu_{\ell},\ldots,Y_{0}=\mu_{0})\\
&  \ \qquad\qquad\times{\mathbb{P}}({\mathcal{Y}}(t)\in{\mathcal{C}}{\text{
for }}t\in\lbrack0,\ell\lbrack,Y_{\ell}=\mu_{\ell},\ldots,Y_{0}=\mu_{0})\\
&  ={\mathbb{P}}(Y_{\ell+1}=\lambda,{\mathcal{Y}}(t)\in{\mathcal{C}}{\text{
for }}t\in\lbrack\ell,\ell+1[\mid Y_{\ell}=\mu_{\ell})\times{\mathbb{P}%
}({\mathcal{Y}}(t)\in{\mathcal{C}}{\text{ for }}t\in\lbrack0,\ell
\lbrack,Y_{\ell}=\mu_{\ell},\ldots,Y_{0}=\mu_{0}).
\end{align*}
We therefore obtain
\begin{align*}
N_{\ell}  &  ={\mathbb{P}}(E\mid Y_{0}=\lambda)\times{\mathbb{P}}(Y_{\ell
+1}=\lambda,{\mathcal{Y}}(t)\in{\mathcal{C}}{\text{ for }}t\in\lbrack\ell
,\ell+1]\mid Y_{\ell}=\mu_{\ell})\\
&  \ \qquad\qquad\qquad\qquad\qquad\times{\mathbb{P}}({\mathcal{Y}}%
(t)\in{\mathcal{C}}{\text{ for }}t\in\lbrack0,\ell\lbrack,Y_{\ell}=\mu_{\ell
},\ldots,Y_{0}=\mu_{0}).
\end{align*}
A similar computation yields
\[
D_{\ell}={\mathbb{P}}(E\mid Y_{\ell}=\mu_{\ell}]\times{\mathbb{P}%
}[{\mathcal{Y}}(t)\in{\mathcal{C}}{\text{ for }}t\in\lbrack0,\ell
\lbrack,Y_{\ell}=\mu_{\ell},\ldots,Y_{0}=\mu_{0}).
\]
Finally, we get
\begin{align*}
{\mathbb{Q}}(Y_{\ell+1}=\lambda\mid Y_{\ell}=\mu_{\ell},\cdots,Y_{0}=\mu_{0})
&  =\\
\qquad\qquad{\mathbb{P}}(Y_{\ell+1}=\lambda,{\mathcal{Y}}(t)\in{\mathcal{C}}
&  {\text{ for }}t\in\lbrack\ell,\ell+1]\mid Y_{\ell}=\mu_{\ell})\times
\frac{{\mathbb{P}}(E\mid Y_{0}=\lambda)}{{\mathbb{P}}(E\mid Y_{0}=\mu_{\ell}%
)}.
\end{align*}

\bigskip

\noindent Laboratoire de Math\'{e}matiques et Physique Th\'{e}orique (UMR CNRS
7350)\newline Universit\'{e} Fran\c{c}ois-Rabelais, Tours \newline
F\'{e}d\'{e}ration de Recherche Denis Poisson - CNRS\newline Parc de
Grandmont, 37200 Tours, France. \newline

\noindent{cedric.lecouvey@lmpt.univ-tours.fr}\newline%
{emmanuel.lesigne@lmpt.univ-tours.fr}\newline{marc.peigne@lmpt.univ-tours.fr}

\end{document}